\documentclass[11pt]{ip-journal}

\usepackage{hyperref}

\usepackage{amssymb}
\usepackage{amsmath}
\usepackage{amsthm}
\usepackage{enumerate}
\usepackage{graphicx}
\usepackage{tikz}
\usetikzlibrary{arrows}

\theoremstyle{plain}
\newtheorem{thm}{Theorem}[section]

\newtheorem{lem}[thm]{Lemma}

\newtheorem{claim}[thm]{Claim}

\numberwithin{equation}{section}

\theoremstyle{definition}
\newtheorem{defn}[thm]{Definition}

\theoremstyle{remark}
\newtheorem*{remark}{Remark}

\newtheorem*{acknowledgements}{Acknowledgements}

\newcommand{\what}{\widehat}
\newcommand{\wtilde}{\widetilde}
\newcommand{\RR}{\mathbb{R}}

\newcommand{\HH}{\mathbb{H}}
\newcommand{\CC}{\mathbb{C}}
\newcommand{\CHAT}{\what{\mathbb{C}}}

\newcommand{\g}{\mathbf{g}} %genus
\newcommand{\p}{\mathbf{p}} %number of punctures
\newcommand{\geod}{\mathcal{G}} %geodesic flow

\newcommand{\eps}{\varepsilon}
\newcommand{\scc}{\mathcal{C}_0} %simple closed curves
\newcommand{\mc}{\mathcal{C}} %multicurves

\newcommand{\neigh}{\mathcal{N}} %neighborhood

\DeclareMathOperator{\re}{Re}

\DeclareMathOperator{\teich}{\mathcal{T}}
\DeclareMathOperator{\el}{EL} %extremal length
\DeclareMathOperator{\psl}{PSL}
\DeclareMathOperator{\cat}{CAT}
\DeclareMathOperator{\gl}{GL}
\DeclareMathOperator{\flt}{\mathcal{QD}} %quadratic differentials
\DeclareMathOperator{\area}{area}
\DeclareMathOperator{\Dil}{Dil}

\begin{document}

\title{Non-convex balls in the Teichm\"uller metric}

\author{Maxime Fortier Bourque}
\address{Department of Mathematics, University of Toronto, 40 St. George Street, Toronto, ON, Canada M5S 2E4}
\email{mbourque@math.toronto.edu}

\author{Kasra Rafi}
\address{Department of Mathematics, University of Toronto, 40 St. George Street, Toronto, ON, Canada M5S 2E4}
\email{rafi@math.toronto.edu}


\begin{abstract}
We prove that the Teichm\"uller space of surfaces of genus $\g$ with $\p$ punctures contains balls which are not convex in the Teich\-m\"ul\-ler metric whenever its complex dimension $(3\g-3+\p)$ is greater than $1$.
\end{abstract}

\maketitle

\section{Introduction}

Let $\overline S$ be a closed oriented surface and $P\subset \overline S$ a finite set. The Teich\-m\"uller space of $S=\overline S \setminus P$ is the set of conformal structures on $\overline S$ up to biholomorphisms homotopic to the identity rel $P$. The Teichm\"uller metric on this space $\teich(S)$ measures how much diffeomorphisms of $\overline S$ homotopic to the identity rel $P$ must distort angles with respect to different conformal structures. This metric is complete, uniquely geodesic, Finsler, and agrees with the Kobayashi metric on $\teich(S)$. However, its local geometry is quite subtle. Indeed, we prove that:

\begin{thm} \label{thm:mainthm}
There exist non-convex balls in $\teich(S)$ whenever its complex dimension is greater than $1$. %$\dim_\CC \teich(S) > 1$. 
\end{thm}

Note that for any $X \in \teich(S)$, the balls of sufficiently small radius centered at $X$ are convex. This is true in any Finsler manifold \cite{Whitehead} \cite{Traber}. 

%The non-convex balls we exhibit have large radius and are centered in the thin part of Teichm\"uller space where a curve on $X$ is short.

\subsection*{Motivation}

If $S$ is a once-punctured torus or a four-times-punctured sphere, then $\teich(S)$ is isometric to $\HH^2$, the hyperbolic plane with constant curvature $-4$.  This led Kravetz to argue that in general, $\teich(S)$ is non-positively curved in the sense of Busemann \cite{Kravetz}. However, Linch \cite{Linch} found a flaw in Kravetz's reasoning and soon after, Masur \cite{Masur} showed that the result was false: there exist distinct geodesic rays starting from the same point in $\teich(S)$ and staying a bounded distance apart whenever $\dim_\CC \teich(S) > 1$. In particular, Teichm\"uller space is not $\cat(0)$ nor $\delta$--hyperbolic. 

In any proper geodesic metric space $\mathbb{X}$, we have the implications
\begin{align*}
& \mathbb{X} \text{ is non-positively curved in the sense of Busemann} \\
\Rightarrow \quad &  \text{the distance to any point is strictly convex along any} \\
& \text{geodesic not containing that point}\\
\Rightarrow \quad &  \text{closed balls are strictly convex} \\
\Rightarrow  \quad &  \text{the convex hull of any finite set is compact.} 
\end{align*}
%\text{the distance to a point is strictly convex along geodesics not containing that point} \\

The question of whether the third statement held for Teichm\"uller space was originally motivated by the Nielsen realization problem, which Kravetz thought he had solved with his erroneous result. Masur's paper \cite{Masur} rendered the problem open again. If balls had been strictly convex, then a po\-si\-tive solution to the Nielsen realization problem would have followed imme\-dia\-tely. In light of Theorem \ref{thm:mainthm}, this approach was doomed to fail. Thankfully, Kerckhoff solved Nielsen's problem many years ago using the convexity of hyperbolic length along earthquake paths \cite{Kerckhoff2}. See also \cite{Wolpert2} for a solution using Weil--Petersson geometry. Whether the fourth statement holds for Teichm\"uller space is an open question of Masur.

\subsection*{Sketch of proof}

Given $X \in \teich(S)$ and a simple closed curve $\alpha \subset S$, the extremal length $\el(\alpha,X)$ is the smallest $c$ such that a Euclidean cylinder of height $1$ and circumference $c$ embeds conformally in $X$ in the homotopy class of $\alpha$. Similarly, the extremal length of a multicurve is the least possible sum of circumferences of disjoint embedded cylinders of height $1$ (see Section \ref{sec:prelim}). The first step of the proof of Theorem \ref{thm:mainthm} is to reduce it to a statement about extremal length.   

\begin{lem} \label{lem:localmax}
If every ball in $\teich(S)$ is convex, then for every multi\-curve $\gamma \subset S$ and every Teichm\"uller geodesic $t \mapsto Z_t$ in $\teich(S)$ the function $t \mapsto \el(\gamma, Z_t)$ has no local maximum.  
\end{lem}

It was shown in \cite{Lenzhen} that extremal length of a curve is not ne\-ces\-sarily convex along Teichm\"uller geodesics. Indeed, the authors of that paper constructed an example where the function $t \mapsto \el(\alpha, X_t)$ increases by a definite amount at first and then stays nearly constant on a later interval. The idea of our construction is to take such a pair $(\alpha,X_t)$ with the surface having a puncture, then another copy $(\beta,Y_t)$ of the same curve and surface but where the time parameter has been reversed and shifted, and to form a connected sum $Z_t = X_t \# Y_t$ via a small slit at the puncture. This is done in such a way that $t \mapsto Z_t$ is still a Teichm\"uller geodesic.

We then show that $\el(\alpha+\beta,Z_t)$ converges to $\el(\alpha, X_t)+\el(\beta, Y_t)$ as the size of the slit shrinks. If we arrange the time parameter of $Y_t$ so that $\el(\beta, Y_t)$ is nearly constant when $\el(\alpha, X_t)$ increases, and decreases when $\el(\alpha,X_t)$ is nearly constant, then their sum increases on the first interval and decreases on the second interval. By the convergence of $\el(\alpha+\beta,Z_t)$ to the sum, that quantity also increases during the first interval and decreases later, provided that the slit is small enough. This forces a local maximum in between, and thereby proves the existence of a non-convex ball.

This proof requires the surface $S$ to be the connected sum of two surfaces each of which is sufficiently complicated. It does not work when the complex dimension of $\teich(S)$ is less than $4$. For those lower complexity cases, our proof is based on rigorous numerical calculations.

\subsection*{Related results}

In \cite{Lenzhen}, Lenzhen and Rafi proved that balls in $\teich(S)$ are quasi-convex. More precisely, they showed that there exists a constant $c=c(S)$ such that for any ball $B\subset \teich(S)$, every geodesic segment with endpoints in $B$ stays within distance $c$ of $B$. In other words, balls cannot fail to be convex arbitrarily badly.

Theorem \ref{thm:mainthm} indicates that the Teichm\"uller metric is positively curved locally, where balls fail to be convex. There are also large-scale manifestations of positive curvature. Namely, there are unbounded regions in Teichm\"uller space where the Teichm\"uller metric looks like a sup metric on a product \cite{Minsky}. On the other hand, there is a sense in which $\teich(S)$ is hyperbolic relative to its thin parts \cite{MasurMinsky}. We refer the reader to \cite{MasurSurvey} for a survey on curvature aspects of the Teichm\"uller metric and to \cite{Rafi} for a coarse description of the Teichm\"uller metric and its geodesics.

Lastly, Theorem \ref{thm:mainthm} should be put in contrast with previous convexity results:
\begin{itemize}
\item $\teich(S)$ is holomorphically convex \cite{Bers};
\item hyperbolic length of a curve is convex along earthquake paths \cite{Kerckhoff2} and Weil--Petersson geodesics \cite{Wolpert2};
\item hyperbolic length \cite{Wolpert} and extremal length \cite{Miyachi} of a curve are log-plurisubharmonic.
\end{itemize}

\subsection*{Organization}

Section \ref{sec:prelim} starts with some background on Teichm\"uller theory. We then reformulate of the convexity problem in terms of extremal length in Section \ref{sec:horoballs}.  Section \ref{sec:conv} proves the convergence of extremal length under pinching deformations. Examples of local maxima for extremal length are constructed in Section \ref{sec:examples} for surfaces with enough topology. Finally, Section \ref{sec:lshapes} presents the numerical results which settle the lower complexity cases. 

\begin{acknowledgements}
The authors thank Jeremy Kahn for sugges\-ting the proof of Lemma \ref{lem:union}, Curtis McMullen for useful comments, and Vincent Delecroix and David Dumas for advice on computer-assisted proofs. The first author was partially supported by the Fonds de recher\-che du Qu\'ebec -- Nature et technologies. The second author was partially supported by NSERC grant \# 435885.
\end{acknowledgements}

\section{Preliminaries} \label{sec:prelim}

\subsection*{Teichm\"uller space}

A point in Teichm\"uller space $\teich(S)$ is a \emph{marked analytically finite complex structure} on $S$. This means a Riemann surface $X$ together with an orientation-preserving homeomorphism $f: S \to X$ which extends to a homeomorphism $\overline f : \overline S \to \overline X$, where $\overline X$ is a closed Riemann surface containing $X$. Two points $(X,f)$ and $(Y,g)$ are identified if there exists a conformal isomorphism $h:X \to Y$ homotopic to $g\circ f^{-1}$. We will write $X \in \teich(S)$, keeping the marking $f$ implicit. 

\subsection*{Teichm\"uller distance}

A linear map $\RR^2 \to \RR^2$ is \emph{$K$--quasiconformal} if it preserves signed angles up to a factor $K\geq 1$. Equivalently, a linear map is $K$--quasiconformal if it has positive determinant and sends circles to ellipses with major axis to minor axis ratio at most $K$.  

A homeomorphism between Riemann surfaces is \emph{$K$--quasiconformal} if its distributional partial derivatives are locally square-integrable and if its matrix of partial derivatives is $K$--quasiconformal almost everywhere. The \emph{dilatation} $\Dil(h)$ of a quasiconformal homeomorphism $h$ is the smallest $K$ for which it is $K$--quasiconformal. All quasiconformal homeomorphisms considered in this paper will be piecewise smooth.

Given $X$ and $Y$ in $\teich(S)$ with markings $f$ and $g$, the \emph{Teichm\"uller distance} between them is defined as
$$
d(X,Y) = \inf_h \frac12 \log \Dil(h)
$$
where the infimum is taken over all quasiconformal homeomorphisms $h: X \to Y$ homotopic to $g\circ f^{-1}$.

\subsection*{Half-translation structures}

A \emph{half-translation} in $\CC$ is either a translation or a rotation of angle $\pi$ about a point, i.e., a map of the form $z \mapsto \pm z + c$. A \emph{half-translation surface} $\Phi$ is a collection of polygons in $\CC$ with sides identified in pairs via half-translations, with at most finitely many points removed. The Euclidean metric descends to a metric on $\Phi$, which is flat except perhaps at finitely many singularities where the cone angle is a positive integer multiple of $\pi$. We require that there be no $\pi$--angle cone points, i.e., if such singularities arise, they should be removed. This is to make the surface non-positively curved.

A \emph{half-translation structure on $S$} is an orientation-preserving homeo\-morphism $f: S \to \Phi$ where $\Phi$ is a half-translation surface. Two half-translation structures $f:S\to \Phi$ and $g:S \to \Psi$ are equivalent if there is an isometry $h: \Phi \to \Psi$ homotopic to $g \circ  f^{-1}$ which preserves the horizontal direction.

There is a natural projection $\pi$ from the space $\flt(S)$ of half-transla\-tion structures on $S$ to $\teich(S)$ since half-translation structures are in parti\-cu\-lar complex structures. A \emph{half-translation structure on a Riemann surface} $X \in \teich(S)$ is one that projects to $X$ under $\pi$. The set $\pi^{-1}(X)$ of half-translation structures on $X$ is in bijection with the set of non-zero integrable holomorphic quadratic differentials on $X$. Given a quadratic differential $q$ on $X$, one obtains half-translation charts by integrating the $1$--form $\sqrt{q}$. Conversely, given a half-translation structure, the differential $dz^2$ in $\CC$ descends to a holomorphic quadratic differential on the underlying Riemann surface. See \cite{Strebel} for the definition and basic properties of quadratic differentials. We will switch back and forth between the two terminologies as convenient.

\subsection*{Teichm\"uller flow}

The group $\gl^+(2,\RR)$ of orientation-preserving linear automorphisms of $\RR^2$ acts on $\flt(S)$ since it conjugates the group of half-translations to itself. For every $t \in \RR$, the linear map
$$
\geod_t = \begin{pmatrix} e^t & 0 \\ 0 & e^{-t} \end{pmatrix}
$$ is $e^{2t}$--quasiconformal. The action of the diagonal subgroup $\left\{ \geod_t \,\mid\, t \in \RR  \right\}$ on $\flt(S)$ is called the \emph{Teich\-m\"uller flow}. A \emph{Teichm\"uller line} is the projection to $\teich(S)$ of the $\geod_t$--orbit of a half-translation structure $\Phi$, parametrized by $t \mapsto \pi\left( \geod_t  \Phi  \right)$.

Teichm\"uller proved that every Teichm\"uller line is a distance-minimi\-zing geo\-desic for the Teichm\"uller distance. He also proved that through any two distinct points in $\teich(S)$ passes a unique Teichm\"uller line.  

\subsection*{Extremal length}

A \emph{conformal metric} on a Riemann surface $X$ is a Borel mea\-su\-ra\-ble func\-tion $\rho : TX \to \RR_{\geq 0}$ such that $\rho(\lambda v) = |\lambda| \rho(v)$ for every $\lambda \in \CC$ and every tangent vector $v \in TX$. In other words, it is a choice of scale at each point.

Let $\Gamma$ be a set of $1$--manifolds in a Riemann surface $X$. The length of the set $\Gamma$ with respect to a conformal metric $\rho$ is
$$
\ell_\rho(\Gamma) = \ell(\Gamma,\rho) = \inf_{\gamma \in \Gamma} \int_\gamma \rho
$$
and the area of $\rho$ is $\int_X \rho^2$. The \emph{extremal length} of $\Gamma$ in $X$ is defined as
\begin{equation} \label{eq:ELdef}
\el(\Gamma,X) = \sup_\rho \frac{\ell_\rho(\Gamma)^2}{\area(\rho)}
\end{equation}
where the supremum is over all conformal metrics $\rho$ of finite positive area. 

Typically, one takes $\Gamma$ to be the free homotopy class of a simple closed curve $\alpha$ in $X$. We will abuse notation and write length or extremal length of a curve to mean the length or extremal length of its homotopy class. The basic example is when $X$ is an upright Euclidean cylinder of circumference $c$ and height $h$, and $\alpha$ is the curve wrapping once around $X$. In this case, the optimal metric $\rho$ is the Euclidean one and we get that $\el(\alpha,X) = c/h.$ We will write $\el(X)$ instead of $\el(\alpha,X)$ since the core curve $\alpha$ is unique up to homotopy.

Pulling-back metrics shows that extremal length does not increase under conformal embeddings. Thus if $X$ is any Riemann surface and $C \subset X$ is an embedded cylinder, then $\el(C)\geq \el(\alpha,X)$ where $\alpha$ is the core curve in $C$. If $X$ is analytically finite and $\alpha$ is \emph{essential}, meaning that it is not homotopic to a point or a puncture in $X$, then the equality $\el(C)=\el(\alpha,X) $ is achieved for a unique embedded annulus $C \subset X$ homotopic to $\alpha$. Furthermore, there exists a unique half-translation structure $\Phi\in \pi^{-1}(X)$ in which $C$ is an upright Euclidean cylinder of height $1$ and the equa\-lity $\el(\alpha, X)=\ell_\rho(\alpha)^2 / \area(\rho)$ holds if and only if $\rho$ is equal almost everywhere to a positive constant multiple of the Euclidean metric on $\Phi$. These results are due to Jenkins \cite{Jenkins}.

Let $\scc(S)$ be the $0$--skeleton of the curve complex of $S$, i.e., the set of homotopy classes of essential simple closed curves in $S$. Kerckhoff's formula \cite{Kerckhoff} states that for any two points $X, Y \in \teich(S)$ we have
\begin{equation} \label{eq:Kerckhoff}
d(X,Y) = \sup_{\alpha \in \scc(S)} \frac{1}{2} \log \frac{\el(\alpha,Y)}{\el(\alpha,X)}.
\end{equation}

That the Teichm\"uller distance is at least as large as the right-hand side follows from the fact that extremal length does not increase by more than a factor $K$ under $K$--quasiconformal homeomorphisms. This property (for all sets of curves $\Gamma$) is often taken as the definition of quasiconformal maps.

\subsection*{Multicurves}

A \emph{(weighted) multicurve} in $S$ is a formal positive linear combination of essential simple closed curves on $S$ that are pairwise disjoint and pairwise homotopically distinct. The set of homotopy classes of multicurves in $S$ will be denoted $\RR_+\times\mc(S)$; it is the cartesian product of $\RR_+$ with the curve complex.

The length of a multicurve $\alpha = \sum_{j\in J} w_j \alpha_j$ with respect to a conformal metric $\rho$ is the weighted sum of the lengths of its components:
$$
\ell_\rho(\alpha)=\sum_{j\in J} w_j \ell_\rho(\alpha_j).
$$
The definition \ref{eq:ELdef} of extremal length as a supremum of length squared divided by area then extends verbatim to multicurves. 

Given $X \in \teich(S)$ and a multicurve $\alpha = \sum_{j\in J} w_j \alpha_j$, we also have
\begin{equation} \label{eq:ELisinf}
\el(\alpha, X) = \inf \sum_{j\in J} w_j^2 \el(C_j)
\end{equation}
where the infimum is taken over all collections of cylinders $C_j$ embedded conformally and disjointly in $X$, with $C_j$ homotopic to $\alpha_j$ \cite{Dylan}. Again, the infimum is achieved by a unique collection of cylinders $C_j$ and there is a half-translation structure on $X$ in which each $C_j$ is foliated by horizontal trajectories and has height $w_j$ \cite{Renelt}. Such a half-translation structure obtained by gluing cylinders along their boundaries is known as a \emph{Jenkins-Strebel differential}.  

There is a topology on weighted multicurves defined using intersection numbers. With respect to this topology, weighted simple closed curves are dense, and for every $X \in \teich(S)$ the map $\alpha \mapsto \el(\alpha, X)$ is con\-ti\-nuous. In fact, Kerckhoff showed that this map extends continuously to all measured foliations \cite{Kerckhoff}.

\section{Horoballs} \label{sec:horoballs}

The goal of this section is to rephrase the problem of the convexity of balls in terms of extremal length. To this end, we look at sublevel sets of extremal length functions, which we call \emph{horoballs}.

\begin{defn}
Given $\alpha \in \scc(S)$ and $c>0$, we define the associated \emph{open horoball} as
$$
H(\alpha, c) = \{\, X \in \teich(S) : \el(\alpha, X) < c \,\}
$$
and the associated \emph{closed horoball} as
$$
\overline{H}(\alpha, c) = \{\, X \in \teich(S) : \el(\alpha, X) \leq c \,\}.
$$
\end{defn}

\begin{remark}
One can define horoballs for any measured foliation. We emphasize that we only consider horoballs associated with simple closed curves here. 
\end{remark}

Note that the closure of an open horoball is the corresponding closed horoball, and the interior of a closed horoball is the corresponding open horoball. This follows from the fact that the extremal length of a curve $\alpha$ is continuous in the second variable and does not have local minima in $\teich(S)$. Indeed, every point $X \in \teich(S)$ lies on a geodesic along which the extremal length of $\alpha$ increases exponentially, given by the Jenkins-Strebel differential on $X$ with a single cylinder homotopic to $\alpha$. In fact, the boundary of any horoball is a hypersurface in $\teich(S)$ homeomorphic to Euclidean space \cite{GardinerMasur}.

It follows directly from Kerckhoff's formula that closed balls are intersections of closed horoballs.

\begin{lem} \label{lem:intersect}
Every closed ball in $\teich(S)$ is a countable intersection of closed horo\-balls.
\end{lem}

\begin{proof}

Let $X \in \teich(S)$, let $r\geq 0$ and let $\overline{B}(X,r)$ be the closed ball of radius $r$ centered at $X$. By equation \ref{eq:Kerckhoff},  $d(X,Y) \leq r$ if and only if $$\el(\alpha, Y) \leq e^{2r}\el(\alpha,X)$$ for every $\alpha \in \scc(S)$, which shows that
$$
\overline{B}(X,r)= \bigcap_{\alpha \in \scc(S)} \overline{H}(\alpha,e^{2r} \el(\alpha,X)).
$$
\end{proof}

In hyperbolic space, horoballs are limits of larger and larger balls with centers escaping to infinity. More precisely, any open horoball is the union of all the open balls that share a given normal vector. The same description holds in Teichm\"uller space.     

\begin{lem} \label{lem:union}
Every open horoball in $\teich(S)$ is a nested union of open balls.
\end{lem}

\begin{proof}
Let $\alpha \in \scc(S)$ and let $c > 0$. Pick an arbitrary point $X \in \partial H(\alpha,c)$ and consider the geodesic $X_t$ defined by the half-translation structure $\Phi$ on $X$ in which almost all \emph{vertical} trajectories are homotopic to $\alpha$, so that $\el(\alpha,X_t)=e^{-2t}c$. We will show that
$$
H(\alpha,c) = \bigcup_{t>0} B(X_t,t).
$$

If $d(X_t,Y)< t$ then $$\el(\alpha,Y) < e^{2t} \el(\alpha,X_t) = c$$ by Kerckhoff's formula, which shows that $B(X_t,t) \subset H(\alpha,c)$ for every $t>0$. The triangle inequality implies that $B(X_s,s) \subset B(X_t,t)$ whenever $0<s<t$, i.e., the union is nested.

Let $Y \in H(\alpha,c)$ and let $b=\el(\alpha,Y)$. We need to show that $Y \in B(X_t,t)$ when $t$ is large enough, which amounts to constructing a $K$--quasiconformal homeo\-morphism between $Y$ and $X_t$ with $K < e^{2t}$. The construction is essentially the same as the one in \cite{Masur} showing that certain geodesic rays in Teichm\"uller space stay a bounded distance apart. 

Let $Y_s$ be the geodesic through $Y$ corresponding to the half-translation structure in which almost all vertical trajectories are homotopic to $\alpha$, but with the time parameter shifted so that $\el(\alpha,Y_s)=e^{-2s}c$. Then fix an $s<0$ such that $$\frac{c}{b}-e^{2s} > 1.$$ 

Let $\alpha_{Y}$ (respectively $\alpha_{X}$) be the closed vertical trajectory in the half-translation structure on $Y_s$ (respectively $X_s$) that splits the $\alpha$--cylinder in two equal parts. By \cite[Lemma 2]{Masur}, there exists a  quasiconformal homeomorphism $f: Y_s \to X_s$ that respects the markings and sends $\alpha_{Y}$ isometrically onto $\alpha_{X}$. Let $L$ be the quasiconformal dilatation of $f$.

We rescale the flat metric on $Y$ and $Y_s$ so that the circumference of the $\alpha$--cylinder is $1$. After rescaling, the Teichm\"uller map $Y_s \to Y$ becomes a horizontal stretch by some factor bigger than $1$. Now $Y_s$ is a cylinder with boundary identifications, and stretching a cylinder lengthwise is the same as cutting it in the middle and inserting another piece of cylinder to make it longer. In other words, $Y$ can be obtained by cutting $Y_s$ open along the core curve $\alpha_{Y}$ and gluing back a cylinder of modulus\footnote{The modulus of a Euclidean cylinder is the reciprocal of its extremal length, i.e., the distance between its boundary components once it has been rescaled to have circumference $1$.} $$\frac{1}{\el(\alpha,Y)} - \frac{1}{\el(\alpha,Y_s)}=\frac{1}{b} - \frac{e^{2s}}{c}$$ without twisting. Similarly $X_t$ can be obtained from $X_s$ by inserting a cylinder of modulus $$\frac{1}{\el(\alpha,X_t)} - \frac{1}{\el(\alpha,X_s)}= \frac{e^{2t}}{c}-\frac{e^{2s}}{c}$$ in the middle.

From this cut-and-paste decomposition of $Y$ and $X_t$, we can define a marking-preserving homeomorphism $g_t: Y \to X_t$ by using $f$ on the complement of the middle cylinder and using the horizontal stretch map of magnitude
$$
 \frac{e^{2t}-e^{2s}}{\frac{c}{b} - e^{2s}} < e^{2t}-e^{2s} < e^{2t}
$$ 
on the middle cylinder. Then $g_t$ is $K_t$--quasiconformal where $$K_t= \max\left\{\frac{e^{2t}-e^{2s}}{\frac{c}{b} - e^{2s}}, L \right\}.$$ If $t$ is large enough, then $L < e^{2t}$ and hence $d(X_t, Y)\leq \frac12 \log K_t < t$.

\end{proof}

The previous lemmata imply that the convexity of balls is equivalent to the convexity of horoballs.

\begin{thm} \label{thm:equiv}
The following are equivalent in $\teich(S)$:
\begin{itemize}
\item every closed horoball is convex;
\item every closed ball is convex;
\item every open ball is convex;
\item every open horoball is convex.
\end{itemize}
\end{thm}

\begin{proof}
If every closed horoball is convex, then every closed ball is convex by Lemma \ref{lem:intersect}, since an arbitrary intersection of convex sets is convex. 

If every closed ball is convex, then every open ball is convex. Indeed, if an open ball is not convex, then there is a smaller closed ball with the same center which is non-convex.

If every open ball is convex, then every open horoball is convex by Lemma \ref{lem:union}, since nested unions of convex sets are convex. 

Suppose that a closed horoball is non-convex. Then the open horo\-balls of slightly higher level for the same simple closed curve are also non-convex. Thus if every open horoball is convex, then every closed horoball is convex.
\end{proof}

Therefore, to show the existence of a non-convex ball, we need to find a non-convex horoball. More explicitly, we need to find a simple closed curve $\alpha \in \scc(S)$ and three points $X,Y,Z \in \teich(S)$ appearing in that order along a geodesic such that the extremal length of $\alpha$ in $Y$ is strictly larger than in both $X$ and $Z$. We can weaken this criterion by allowing $\alpha$ to be a multicurve and replacing the $3$--point condition by a $4$--point condition. This will be useful later.

\begin{lem} \label{lem:criterion}
Suppose that there exists four points $X,Y,Z,W$ appearing in that order along a geodesic in $\teich(S)$ and a multicurve $\alpha \in \RR_+ \times \mc(S)$ such that $$\el(\alpha, X)<\el(\alpha, Y) \quad \mbox{and} \quad \el(\alpha, Z)>\el(\alpha, W).$$ 
Then there exists a non-convex ball in $\teich(S)$.
\end{lem}
\begin{proof}
Since extremal length depends continuously on the first variable and since weighted simple closed curves are dense in $\RR_+ \times \mc(S)$, there exists a weighted simple closed curve $w \beta$ such that $$\el(w\beta, X)<\el(w\beta, Y) \quad \mbox{and} \quad \el(w\beta, Z)>\el(w\beta, W).$$
As extremal length is homogeneous of degree $2$ in the first variable, we also have $$\el(\beta, X)<\el(\beta, Y) \quad \mbox{and} \quad \el(\beta, Z)>\el(\beta, W).$$

Let $c = \max(\el(\beta, X), \el(\beta, W))$. Then the geodesic segment from $X$ to $W$ has its endpoints in $\overline{H}(\beta, c)$ and passes through $Y$ and $Z$. At least one of $Y$ or $Z$ lies outside $\overline{H}(\beta,c)$, so that the closed horoball $\overline{H}(\beta,c)$ is non-convex. By Theorem \ref{thm:equiv}, this implies the existence of a non-convex ball.
\end{proof}

If extremal length of a multicurve has a strict local maximum along a geodesic, then we can clearly find $4$ points satisfying the hypotheses of the above lemma, so that non-convex balls exist. It turns out that every local maximum is a strict local maximum (see below), which implies Lemma \ref{lem:localmax}. In practice, the hypotheses of Lemma \ref{lem:criterion} are easier to check than the existence of a local maximum, but Lemma \ref{lem:localmax} was simpler to state for the introduction.
  
\begin{lem}
Any local maximum of extremal length along a geodesic is a strict local maximum.
\end{lem}
\begin{proof}
Suppose there is a multicurve $\alpha\in \RR_+ \times \mc(S)$ and a geodesic $t \mapsto X_t$ in $\teich(S)$ such that the function $f(t)=\el(\alpha, X_t)$ has a local maxi\-mum at $T$ and let $M=f(T)$. If the local maximum is not strict, then $f^{-1}(M)$ accumulates at $T$. Since $f$ is real-analytic \cite{Miyachi}, it must be constant by the identity principle. But this is impossible. Indeed, take $\rho_t$ to be the Euclidean metric defining the geodesic $X_t$. Then the area of $\rho_t$ is constant, whereas the length $\ell(\alpha,\rho_t)$ is unbounded in at least one direction. Indeed, $\ell(\alpha,\rho_t)$ is bounded below by the intersection number between $\alpha$ and the vertical foliation on $X_t$ as well as by the intersection number with the horizontal foliation. These intersection numbers depend exponentially on $t$ and at least one of them is non-zero, so that it diverges as $t \to \infty$ or as $t \to -\infty$. Since $f(t) \geq \ell(\alpha,\rho_t)^2 / \area(\rho_t)$, this is a contradiction.
\end{proof} 

Rather than exhibiting $4$ collinear points verifying the inequalities of Lemma \ref{lem:criterion}, we will construct a sequence of collinear points $X_n$, $Y_n$, $Z_n$ and $W_n$ that degenerate in a controlled way as $n \to \infty$ and such that the desired inequalities hold in the limit. We thus need to show that extremal length behaves well under mild degeneration.

\section{Convergence of extremal length under pinching} \label{sec:conv}

Let $R = \sqcup_{j \in J} R_j$ be a subsurface of $S$ obtained by cutting $S$ along a multicurve and possibly forgetting some of the pieces. Each connected component $R_j$ of $R$ is homeomorphic to a punctured surface $R_j'$. The Teichm\"uller space $\teich(R)$ is defined as the Cartesian product $\Pi_{j\in J} \teich(R_j')$. 

\begin{defn} \label{def:conv}
Let $X_n \in \teich(S)$ and $Y \in \teich(R)$. We say that \emph{$X_n$ converges conformally to $Y$ as $n\to \infty$} if there exist nested surfaces $Y_n \subset Y$ exhausting $Y$ and $K_n$--quasiconformal embeddings $f_n:Y_n \to X_n$ homotopic to the inclusion map $R \subset S$ such that $K_n \to 1$ as $n\to \infty$.
\end{defn}

We emphasize that in this definition, the ends of $Y$ are all required to be punctures. Informally speaking, $X_n$ converges conformally to $Y$ if there is a multicurve in $X_n$ that gets pinched and in the process, some pieces survive to form $Y$. We don't care about the other pieces, meaning that they don't need to stabilize as $n \to \infty$. Thus conformal convergence in the above sense is more general than convergence in the augmented Teichm\"uller space, where every piece is required to stabilize \cite{Abikoff}.

There are different equivalent ways to formulate conformal convergence. We can say that $X_n$ converges conformally to $Y$ if 
\begin{itemize}
\item for every simple closed curve $\alpha$ in $R$ (including peripheral ones), the hyperbolic length of $\alpha$ in $X_n$ converges to the hyperbolic length of $\alpha$ in $Y$;
\item for every $j\in J$, the covering space of $X_n$ associated with the subsurface $R_j$, equipped with its hyperbolic metric, converges in the Gromov-Hausdorff topology to the corresponding component $Y_j$ of $Y$ with respect to some choices of basepoints;  
\item  for every $j\in J$, $\rho_n^j$ converges up to conjugacy to $\rho^j$, where $\rho_n:\pi_1(S) \to \psl(2,\RR)$ is the representation defining $X_n$, $\rho_n^j$ is its restriction to $\pi_1(R_j)$ coming from the inclusion $R_j \subset S$ and $\rho_j:\pi_1(R_j) \to \psl(2,\RR)$ is the representation defining $Y_j$.
\end{itemize}

For our purposes, the definition in terms of nearly conformal embeddings is the most convenient. The statement we want to prove is that conformal convergence implies convergence of extremal length for multicurves supported on the limiting surface. Extremal length on a disconnected surface is defined in the usual way, as the supremum of weighted length squared divided by area over all conformal metrics. A standard argument shows this equals the sum of the extremal lengths on connected components (cf. \cite[p. 55]{Ahlfors}).

\begin{lem} \label{lem:extadditive}
Let $Y = \sqcup_{j \in J} Y_j$ be a disjoint union of Riemann surfaces, let $\alpha$ be a multi\-curve on $Y$ and let $\alpha^j=\alpha \cap Y_j$. Then $$\el(\alpha, Y) = \sum_{j \in J} \el(\alpha^j, Y_j).$$
\end{lem}
\begin{proof}
First observe that the extremal length of $\alpha$ on $Y$ is the same as its extremal length on the union $Z$ of the components which it intersects. Indeed, given a metric $\rho$ on $Y$, the ratio $\ell(\alpha,\rho)^2 / \area(\rho)$ does not decrease if we modify $\rho$ to be zero outside $Z$. This shows that $\el(\alpha,Y) \leq \el(\alpha, Z)$ and the reverse inequality follows by extending any conformal metric on $Z$ to be zero on $Y \setminus Z$. 

If $\alpha^j$ is empty, then clearly $\el(\alpha^j, Y_j)=0$. In proving the above formula, we may therefore assume that $\alpha^j$ is non-empty for each $j \in J$.

For each $j \in J$, let $\rho_j$ be any metric on $Y_j$ such that $\ell(\alpha^j, \rho_j)$ and $\area(\rho_j)$ are finite and positive. By rescaling $\rho_j$, we may assume that $\ell(\alpha^j, \rho_j)=\area(\rho_j)$. Let $\rho$ be the metric on $Y$ which is equal to $\rho_j$ on $Y_j$. Then 
$$\ell(\alpha,\rho)=\sum_{j \in J} \ell(\alpha^j, Y_j) = \sum_{j \in J} \area(\rho_j)=\area(\rho)$$
which implies that
$$
\el(\alpha, Y) \geq \frac{\ell(\alpha,\rho)^2}{\area(\rho)} = \sum_{j \in J} \frac{\ell(\alpha,\rho_j)^2}{\area(\rho_j)}.
$$
We can replace the right-hand side by its supremum over all non-degenerate metrics $\rho_j$ to get 
$$\el(\alpha, Y) \geq \sum_{j \in J} \el(\alpha^j, Y_j).$$

Conversely, let $\sigma$ be any conformal metric on $Y$ and let $\sigma_j$ be its restriction to $Y_j$. Then for each $j\in J$ we have
$$
 \el(\alpha^j, Y_j) \geq \frac{\ell(\alpha^j, \sigma_j)^2}{\area(\sigma_j)}.
$$ 
Summing over all $j$ yields 
$$
\sum_{j\in J}\el(\alpha^j, Y_j) \geq 
\sum_{j \in J} \frac{\ell(\alpha^j, \sigma_j)^2}{\area(\sigma_j)} \geq 
\frac{\left(\sum_{j \in J} \ell(\alpha^j, \sigma_j)\right)^2}{\sum_{j \in J} \area(\sigma_j)}
$$
where the second inequality follows from the Cauchy-Schwarz inequality. Finally, observe that $$\sum_{j \in J} \ell(\alpha^j, \sigma_j) = \ell(\alpha,\sigma) \quad \text{and} \quad \sum_{j \in J} \area(\sigma_j) = \area(\sigma)$$ so that
$$
\sum_{j\in J}\el(\alpha^j, Y_j) \geq \frac{\ell(\alpha,\sigma)^2}{\area(\sigma)}.
$$
Since the inequality holds for any conformal metric $\sigma$, it holds for the supremum as well and we have \begin{align*}
\sum_{j\in J}\el(\alpha^j, Y_j) &\geq \el(\alpha, Y). \qedhere
\end{align*}
\end{proof}

This lemma implies that equation \ref{eq:ELisinf}, which says that extremal length is the infimum of weighted sums of extremal lengths of embedded cylinders, still holds for disconnected surfaces. We use this in the proof of convergence of extremal length under degeneration.

\begin{thm} \label{thm:elconv}
Let $\alpha$ be a multicurve in $R$ and suppose that $X_n$ converges conformally to $Y$ as $n\to \infty$. Then $\el(\alpha, X_n) \to \el(\alpha, Y)$ as $n \to \infty$.
\end{thm}

\begin{proof}
Let $K>1$. We will show that if $n$ is large enough, then $$\frac{1}{K^2} \el(\alpha, Y) \leq \el(\alpha, X_n) \leq K^2 \el(\alpha, Y),$$
starting with the second inequality.

Write $\alpha$ as a weighted sum of simple closed curves $\sum_{i\in I} w_i \alpha_i$ and let $C = \sqcup_{i \in I} C_i$ be the collection of cylinders in $Y$ such that $\el(\alpha, Y) = \sum_{i\in I} w_i^2 \el(C_i)$. For each $i \in I$, let $A_i \subset C_i$ be a compactly contained essential cylinder (for example a straight subcylinder) such that $\el(A_i) \leq K \el(C_i)$. 

Let $Y_n$ be a nested exhaustion of $Y$ and let $f_n: Y_n \to X_n$ be quasiconformal embeddings as in Definition \ref{def:conv}. If $n$ is large enough, then $Y_n$ contains $\cup_{i \in I} A_i$ and $f_n$ is $K$--quasiconformal. Then by equation \ref{eq:ELisinf} we have
\begin{align*}
\el(\alpha, X_n) & \leq \sum_{i \in I} w_i^2 \el(f_n(A_i)) \\
 &\leq \sum_{i \in I}  w_i^2 K \el(A_i) \\
 & \leq K^2 \sum_{i \in I}  w_i^2  \el(C_i) = K^2 \el(\alpha, Y). 
\end{align*}

For the reverse inequality, let $\rho$ be the conformal metric realizing $\el(\alpha, Y)$. Our goal is to construct a good enough conformal metric on $X_n$ from $\rho$.

We may assume that $\alpha$ intersects every component of $Y$ (otherwise ignore the superfluous components). By the previous lemma and Renelt's theorem \cite{Renelt}, $\rho=\sqrt{|q|}$ for a holomorphic quadratic differential $q$ on $Y$ with at most simple poles at the punctures. Let $\overline Y$ be the completion of $Y$ in the metric $\rho$ and let $Q = \overline Y \setminus Y$ be the set of punctures of $Y$. Since $Q$ is finite, there exists a $\delta_0>0$ such that the $\delta_0$--balls around the points of $Q$ are embedded and pairwise disjoint. Here we are using the fact that $\rho$ has isolated singularities so that the distance it induces on $\overline Y$ via path integrals is compatible with the underlying topology.

Let $\mu = 2 \delta_0$. Then any homotopically non-trivial arc from $Q$ to itself has $\rho$--length at least $\mu$. For $\delta>0$, let $\neigh_\delta(Q)$ be the open $\delta$--neighborhood of $Q$ in the metric $\rho$ and let $\rho_\delta$ be the metric on $Y$ which agrees with $\rho$ outside of $\overline{\neigh_\delta(Q)}$ and is identically zero on $\overline{\neigh_\delta(Q)} \cap Y$. 

\begin{claim}
For every $\beta \in \scc(R)$ the length $\ell( \beta,\rho_\delta)$ converges to  $\ell(\beta, \rho)$ as $\delta \to 0$.
\end{claim}
\begin{proof}[Proof of claim]
It is clear that $\ell( \beta,\rho_\delta)\leq \ell( \beta,\rho)$ since $\rho_\delta \leq \rho$. We will show that 
$$
\ell( \beta,\rho) \leq \frac{\mu+\delta}{\mu-2\delta }\, \ell( \beta,\rho_\delta)
$$
whenever $\delta < \delta_0$.

Let $\gamma$ be a piecewise smooth curve  homotopic to $\beta$ on $Y$. Our goal is to find a curve $\wtilde \gamma$ homotopic to $\gamma$ such that 
$$
|\wtilde \gamma | \leq \frac{\mu+\delta}{\mu-2\delta }\, \left|\gamma \setminus \overline{\neigh_\delta(Q)}\right|
$$
where the length $|\cdot|$ is measured with respect to $\rho$.

We modify $\gamma$ in two steps, each time making it shorter outside $\overline{\neigh_\delta(Q)}$. We estimate its total length at the end. To avoid unnecessary notation, we denote the modified curve by $\gamma$ again instead of $\wtilde \gamma$. 

We may assume that each component of $\gamma\setminus \overline{\neigh_\delta(Q)}$ is homotopically non-trivial rel endpoints after collapsing each component of $\overline{\neigh_\delta(Q)}$. Otherwise, we can homotope those trivial subarcs to the boundary of $\overline{\neigh_\delta(Q)}$, which shortens $\gamma$. We can also assume that $\gamma \setminus \overline{\neigh_\delta(Q)}$ has only finitely many components, because each such component has length at least $\mu - 2 \delta$. If $\gamma \setminus \overline{\neigh_\delta(Q)}$ has infinite length, then there is nothing to show. Thus $\gamma \cap \overline{\neigh_\delta(Q)}$ has finitely many components as well. Each such component can be homotoped, keeping the endpoints fixed, to a path in $\overline{\neigh_\delta(Q)}$ of $\rho$--length at most $3\delta$. This does not change the length of the portion of $\gamma$ lying outside $\overline{\neigh_\delta(Q)}$.

If $\gamma$ is disjoint from $\overline{\neigh_\delta(Q)}$, then $|\gamma| = \left|\gamma \setminus \overline{\neigh_\delta(Q)}\right|$ and we are done.

Otherwise $\gamma$ breaks up into $\gamma\setminus \overline{\neigh_\delta(Q)}$ and $\gamma\cap \overline{\neigh_\delta(Q)}$ and these two sets have the same number of components. Let $\sigma$ be a component of $\gamma\setminus \overline{\neigh_\delta(Q)}$ and $\tau$ a component of $\gamma\cap \overline{\neigh_\delta(Q)}$. Then $|\sigma| \geq \mu - 2\delta$ and $|\tau| \leq 3 \delta$, so that
$$
|\sigma|+|\tau| \leq \frac{\mu + \delta}{ \mu - 2 \delta} \, |\sigma|.
$$
Adding these inequalities over distinct $\sigma$-$\tau$ pairs whose union is $\gamma$ yields 
$$
|\gamma| \leq \frac{\mu + \delta}{ \mu - 2 \delta}\, \left|\gamma \setminus \overline{\neigh_\delta(Q)} \right|,
$$ 
which shows that 
$$
\ell( \beta,\rho) \leq \frac{\mu+\delta}{\mu-2\delta }\,\left|\gamma \setminus \overline{\neigh_\delta(Q)} \right|.
$$
Since $\gamma$ was arbitrary, the right-hand side can be replaced with the infimum $\frac{\mu+\delta}{\mu-2\delta }\, \ell( \beta,\rho_\delta)$.
\end{proof}

The analogous result for multicurves follows by linearity. Thus there exists a $\delta>0$ such that 
$$
\ell(\alpha,\rho_\delta)^2 \geq \frac{1}{K} \ell(\alpha,\rho)^2.
$$
We fix such a $\delta$ for the rest of the proof.

Let $n$ be large enough so that $Y_n$ contains $Y \setminus \neigh_\delta(Q)$ and so that $f_n:Y_n \to X_n$ is $K$--quasiconformal. We may assume that $f_n$ is smooth except at finitely many points. 

We define a conformal metric $\rho_n$ on $X_n$ by $$\rho_n(v)= \max_{\theta \in [0,2\pi]} \rho(\mathrm{d}f_n^{-1}(e^{i\theta} v))$$
if $v$ is a tangent vector based at a point in $f_n(Y \setminus \neigh_\delta(Q))$ and $\rho_n(v)=0$ otherwise.

\begin{claim}
For every $\beta \in \scc(R)$, we have $\ell( \beta,\rho_n)\geq\ell(\beta, \rho_\delta)$.
\end{claim}
\begin{proof}[Proof of claim]
Let $\gamma$ be a curve homotopic to $\beta$ in $X_n$. If $\gamma$ is contained in the image $f_n(Y \setminus \neigh_\delta(Q))$, then

\begin{align*}
\int \rho_n(\gamma'(t)) \,|dt| &\geq \int \rho((\mathrm{d}f_n^{-1} (\gamma'(t))) \,|dt| \\
& =\int \rho((f_n^{-1}\circ \gamma)'(t))) \,|dt| \\
& \geq \ell(\beta, \rho) \\
& \geq \ell(\beta, \rho_\delta).
\end{align*}

Otherwise, we can homotope $\gamma$ to a curve $\wtilde \gamma$ which is not longer, yet is contained in $f_n(Y \setminus \neigh_\delta(Q))$. 
\end{proof}

Once again, the analogous result for multicurves on $R$ follows immediately.

\begin{claim}
The areas satisfy $\area(\rho_n) \leq K \area(\rho)$.
\end{claim}
\begin{proof}[Proof of claim]
Since $f_n$ is $K$--quasiconformal, we have
$$\max_{\theta \in [0,2\pi]} \rho(\mathrm{d}f_n^{-1}(e^{i\theta} v)) \leq K \min_{\theta \in [0,2\pi]} \rho(\mathrm{d}f_n^{-1}(e^{i\theta} v)).$$

This shows that $\rho_n^2 \leq K (f_n)_* \rho^2$ on $f_n(Y \setminus \neigh_\delta(Q))$, which means that
\begin{align*}
\area(\rho_n) &= \int_{X_n} \rho_n^2  \\ &\leq K \int_{f_n(Y \setminus \neigh_\delta(Q))} (f_n)_* \rho^2 \\ &= K \int_{Y \setminus \neigh_\delta(Q)} \rho^2  \\ &\leq  K \area(\rho).
\end{align*}
\end{proof}

Combining the above inequalities yields
$$
\el(\alpha,X_n) \geq \frac{\ell(\alpha,\rho_n)^2}{\area(\rho_n)} \geq \frac{1}{K} \frac{\ell(\alpha,\rho_\delta)^2}{\area(\rho)} \geq \frac{1}{K^2} \frac{\ell(\alpha,\rho)^2}{\area(\rho)} = \frac{1}{K^2} \el(\alpha, Y)
$$
for all large enough $n$, which is what we wanted to show.
\end{proof}

We now possess all the necessary tools to construct Teichm\"uller geode\-sics along which extremal length of a multicurve increases first and then decreases later.

\section{The examples} \label{sec:examples}

\subsection{The sphere with seven punctures} \label{sec:staircase}

Given non-negative lengths $l_1,\ldots,l_n$ and heights $h_1,\ldots,h_n$, consider the staircase-shaped polygon $P(l_1,h_1,\ldots,l_n,h_n) \subset \RR^2$ with $j$--th step of length $l_j$ and height $h_j$. This is illustrated in Figure \ref{fig:staircase} for $n=4$. We allow either $l_1$ or $h_n$ to be infinite, in which case $P$ is a horizintal or vertical half-infinite strip ending in a staircase. If all lengths and heights are finite, then we put the bottom-left corner of $P$ at the origin. 

\begin{figure}[htbp] 
\centering
\includegraphics{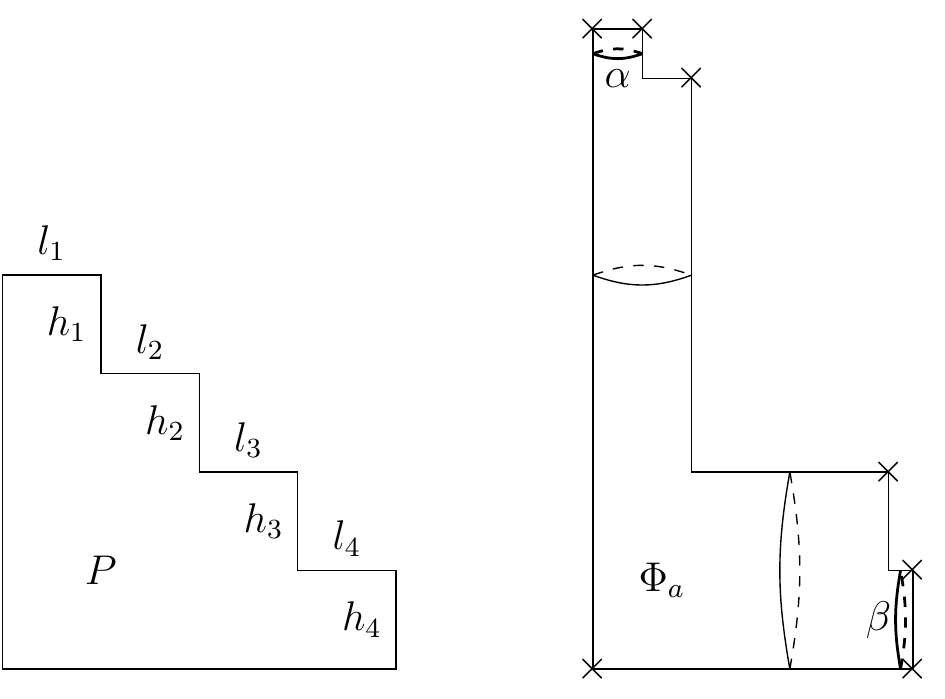}
 
\caption{The staircase-shaped polygon $P$ and the half-trans\-la\-tion sur\-face $\Phi_a$.}\label{fig:staircase}
\end{figure}

For $0<a<1$, let $$P_a=P(1,1,1,1/a^3,1/a^2,1/a,a,1/a)$$ and let $\overline \Phi_a$ be the double of $P_a$ across its boundary. More precisely, take $P_a$ and its image $P_a^*$ under a horizontal reflection $\sigma$ and glue them along their boundary using $\sigma$. It helps to think of $P_a$ as the front of $\overline \Phi_a$ and $P_a^*$ as the back. The sphere $\overline \Phi_a$ has $7$ cone points of angle $\pi$ (coming from the $7$ interior right angles in $P_a$), which we remove in order to get the half-translation surface $\Phi_a$.

Let $\alpha$ be the simple closed curve that separates the two highest punctures on $\Phi_a$ from the rest, i.e., the double of the middle horizontal line in the highest step of the staircase $P_a$. Similarly, let $\beta$ be the double of the middle vertical line in the step furthest to the right in $P_a$ (see Figure \ref{fig:staircase}). 

As we apply Teichm\"uller flow to the half-translation surface $\Phi_a$, the extremal length of $\alpha$ increases rapidly at first and then stays nearly constant for a long time. The extremal length of $\beta$ does the opposite: it remains nearly constant for a long time then decreases rapidly. Also, since $\alpha$ and $\beta$ are separated by cylinders of very large modulus, the extremal length of $\alpha+\beta$ is roughly equal to the sum of the individual extremal lengths. The net effect is that $\el(\alpha+\beta , \geod_t   \Phi_a)$ increases at first and decreases some time later. 

Let us be more precise. Consider the points
\begin{align*}
X_a &= \begin{pmatrix} 1/e & 0 \\ 0 & e \end{pmatrix}\Phi_a &  Y_a &= \Phi_a  \\
 Z_a &= \begin{pmatrix} 1/a & 0 \\ 0 & a \end{pmatrix}\Phi_a &    W_a &= \begin{pmatrix} e/a & 0 \\ 0 & a/e \end{pmatrix} \Phi_a.
\end{align*}
appearing at times $-1$, $0$, $\log(1/a)$ and $\log(1/a)+1$ along the Teichm\"uller line $t \mapsto \geod_t  \Phi_a$.

Observe that $Z_a = \tau  Y_a$ and $W_a = \tau  X_a$, where $\tau$ is the reflection about the line $y=x$ in the plane. Indeed, the definition of $P_a$ was arranged so that
$$\begin{pmatrix} 1/a & 0 \\ 0 & a \end{pmatrix}  P_a =P(1/a,a,1/a,1/a^2,1/a^3,1,1,1) = \tau  P_a.$$

We claim that from the point of view of the multicurve $\alpha + \beta$, the surfaces $X_a$, $Y_a$, $Z_a$ and $W_a$ all have conformal limits as $a\to 0$. Let $\Psi^\alpha$ be the double of $P(1,1,1,\infty)$ minus the $\pi$--angle singularities, let $\Psi^\beta$ be the double of $P(\infty,1,0,1)$ minus the three distinguished vertices, and let $\Psi=  \Psi^\alpha \sqcup \Psi^\beta$ (see Figure \ref{fig:limit}). We consider $\Psi$ as a half-translation surface with infinite area. The conformal limits of $X_a$, $Y_a$, $Z_a$ and $W_a$ as $a\to 0$ are $$X_0 = \geod_{-1} \Psi, \quad Y_0=\Psi, \quad Z_0 = \tau  \Psi \quad \mbox{and} \quad W_0 = \tau  \geod_{-1}  \Psi$$ respectively.

\begin{figure}[htbp] 
\centering
\includegraphics{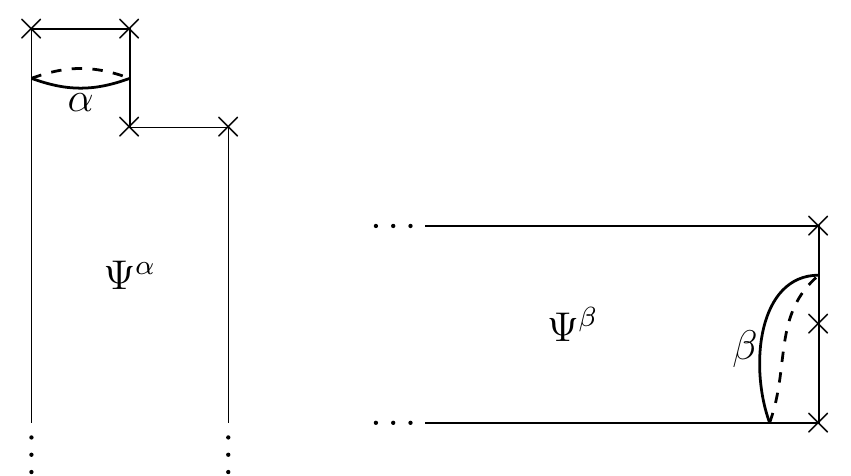}
 
\caption{The conformal limit $\Psi$ at time $t=0$. Observe that $\Psi^\beta$ is conformally invariant under Teichm\"uller flow.}\label{fig:limit}
\end{figure}

\begin{lem} \label{lem:surfconv}
For each $\Lambda\in \{X,Y,Z,W\}$, the surface $\Lambda_a$ converges conformally to $\Lambda_0$ as $a\to 0$.
\end{lem}
\begin{proof}
We prove that $Y_a=\Phi_a$ converges conformally to $Y_0=\Psi$ as $a \to 0$. The proof of conformal convergence at other times is similar.

It is clear that the top half of $P_a$ converges to $P(1,1,1,\infty)$ as $a \to 0$. Indeed, for any $L>0$ the top portion of $P(1,1,1,\infty)$ of height $L$ embeds isometrically into $P_a$ provided that $1/a^2 > L$. Moreover, the isometric embedding maps vertices to vertices. By doubling this embedding, we obtain a conformal embedding of a large portion of $\Psi^\alpha$ into $\Phi_a$.

If we apply an homothety of factor $a$ to $P_a$, we get the polygon $$Q_a=P(a,a,a,1/a^2,1/a,1,a^2,1).$$  We will show that the right half of $Q_a$ converges conformally to the unbounded polygon $P(\infty,1,0,1)$. For this, we will use a theorem of Rad\'o which says that if a sequence of parametrized Jordan curves $\gamma_n: S^1 \to \CHAT$ converges uniformly to a Jordan curve $\gamma_\infty: S^1 \to \CHAT$, then the corresponding (appropriately  normalized) Riemann maps converge uniformly on the closed unit disk \cite[p.26]{Pommerenke} \cite[p.59]{Goluzin}.

For $a\geq 0$, let $R_a=P(\infty,1,a^2,1)$. For concreteness, we take the finite vertices of $R_a$ to be located at $0$, $i$, $-a^2+i$ and $-a^2 + 2i$ in $\CC$. Let $\gamma_a: \RR\cup \{\infty\} \to \CHAT$ be the Jordan curve $\partial R_a \cup \{\infty\}$ parametrized counter-clockwise by arclength with $\gamma_a(0)=0$ and $\gamma_a(\infty)=\infty$. Then $\gamma_a$ converges uniformly to $\gamma_0$ as $a \to 0$.

Let $f_a : R_0 \to R_a$ be the unique conformal homeomorphism such that $f_a(0)=0$, $f_a(2i) = -a^2 + 2i$ and $f_a(\infty)=\infty$. Then $f_a$ converges uniformly to the identity by Rad\'o's theorem. In particular, $f_a^{-1}(i)=i v_a$ converges to $i$ as $a \to 0$. Define  $g_a: R_0 \to R_0$ by

$$
g_a(x+iy) = \begin{cases} x+ i v_a y & \mbox{if } y \in [0,1] \\ x +i ((2-v_a)(y-1)+v_a) & \mbox{if } y\in (1,2].  \end{cases}
$$
This map is piecewise linear, fixes $0$, $2i$ and $\infty$ and sends $i$ to $f_a(i)$. The quasiconformal dilatation of $g_a$ is equal to $\max \left\{ v_a, \frac{1}{v_a}, 2-v_a, \frac{1}{2-v_a} \right\}$, which tends to $1$ as $a \to 0$. Thus when $a$ is small, $f_a\circ g_a$ is a nearly conformal homeomorphism from $R_0$ to $R_a$ taking $0$, $i$, $2i$ and $\infty$ to $0$, $i$, $-a^2+2i$ and $\infty$ respectively.

Let $L<0$ and let $C = \{ z \in R_0  \mid \re z > L \}$. The images $f_a \circ g_a(C)$ stay bounded away from $\infty$ since $f_a \circ g_a$ converges uniformly to the identity. Let's say that $\re z \geq u$ for all  $z \in f_a \circ g_a(C)$ and all $a\geq 0$. Then $f_a \circ g_a(C)$ embeds isometrically in $Q_a$ in the obvious way provided that $1/a > |u|$. By doubling all these objects and maps, we obtain a quasiconformal embedding of the double of $C$ into $\Phi_a$ with dilatation arbitrarily close to $1$ when $a$ is small.      

\end{proof}

It only remains to prove estimates for the extremal length on these limiting surfaces.

\begin{lem} \label{lem:bound1}
$\el(\alpha, X_0)$ and $\el(\beta, W_0)$ are bounded above by  $2/e^{2}.$ 
\end{lem}
\begin{proof}
Recall that the component of $X_0$ containing $\alpha$ is $\geod_{-1} \Psi^\alpha$. Take the top $1\times 1$ square in $P(1,1,1,\infty)$ without its horizontal sides. Its double is an open Euclidean cylinder of circumference $2$ and height $1$ homotopic to $\alpha$ in $\Psi^\alpha$. This cylinder gets stretched to one of circumference $2/e$ and height $e$, hence extremal length $2/e^2$, under the map $\geod_{-1}$. The inequality $\el(\alpha, X_0)\leq 2/e^2$ thus follows from the monotonicity of extremal length under conformal embeddings. 

Now, the reflection $\tau$ maps $X_0$ anti-conformally onto $W_0$ and sends $\alpha$ to $\beta$ so that $\el(\alpha, X_0)=\el(\beta, W_0)$.
\end{proof}

\begin{lem} \label{lem:bound2}
$\el(\alpha,Y_0)$ and $\el(\beta,Z_0)$ are bounded below by $2/3.$ 
\end{lem}
\begin{proof}
Take $\rho$ to be the Euclidean metric on the top part $T$ of height $2$ in the component of $Y_0$ containing $\alpha$ (this is a union of $6$ unit squares, $3$ in the front, $3$ in the back) extended to be identically zero elsewhere. Then $\rho$ has area $6$. Moreover, any curve $\gamma$ homotopic to $\alpha$ on $Y_0$ has length at least $2$ in the metric $\rho$. If $\gamma$ is not contained in $T$, then some point $p$ on $\gamma$ is at height less than $-2$. But some point $q$ on $\gamma$ has to be at height at least $-1$ since it is homotopic to $\alpha$ (it has to cross the seam between the front and back of $Y_0$ joining the two punctures on the top right). In this case, the length of $\gamma$ is at least twice the height difference between $q$ and the bottom of $T$ (because there is a subarc from $p$ to $q$ then from $q$ to $p$), i.e., at least $2$. A similar argument also applies if $\gamma$ is contained in $T$ (it then has to cross the left seam in addition to the other one). We conclude that the extremal length of $\alpha$ on $Y_0$ is at least $2^2 / 6 = 2/3$.

The extremal length of $\beta$ on $Z_0$ is the same as the extremal length of $\alpha$ on $Y_0$ by symmetry.
\end{proof}

\begin{lem} \label{lem:bound3}
We have $\el(\beta, X_0) = \el(\beta, Y_0)$ and $\el(\alpha,Z_0)=\el(\alpha,W_0)$. 
\end{lem}
\begin{proof}
The equality $\el(\beta, X_0) = \el(\beta, Y_0)$ is due to the fact that the component $\geod_{-1}\psi^\beta $ of $X_0$ containing $\beta$ is conformally equivalent to the corresponding component $\Psi^\beta$ of $Y_0$. Indeed, recall that $\Psi^\beta$ is the double of $P(\infty,1,0,1)$. The image of the latter by $\geod_{-1}$ is $P(\infty,e,0,e)$ which is homothetic to the first polygon by a factor $e$. This homothety doubles to a conformal isomorphism between $\Psi^\beta$ and $\geod_{-1}\Psi^\beta$ preserving the marked points and the curve $\beta$.

Similarly, $\el(\alpha,Z_0)=\el(\alpha,W_0)$ since the connected component of $Z_0$ containing $\alpha$ is conformally invariant under Teichm\"uller flow. 
\end{proof}

These are all the ingredients we need to prove the desired behavior for the extremal length of $\alpha + \beta$ along the geodesic $\geod_t  \Phi_a$.

\begin{thm}
If $a$ is small enough, then $$\el(\alpha+\beta, X_a)<\el(\alpha+\beta, Y_a) \quad \mbox{and} \quad \el(\alpha+\beta, Z_a)>\el(\alpha+\beta, W_a).$$
\end{thm}
\begin{proof}
For each $\Lambda\in \{X,Y,Z,W\}$ we have that $$\el(\alpha + \beta , \Lambda_a) \to \el(\alpha + \beta , \Lambda_0)$$ as $a \to 0$ by Lemma \ref{lem:surfconv} and Theorem \ref{thm:elconv}. Since each $\Lambda_0$ is disconnected, we also have
$$
\el(\alpha + \beta , \Lambda_0) = \el(\alpha, \Lambda_0) + \el(\beta, \Lambda_0).
$$
by Lemma \ref{lem:extadditive}.

By the previous three lemmata we have 
\begin{align*}
\el(\alpha+\beta, X_0) &=\el(\alpha, X_0)+\el(\beta,X_0) \\
&\leq \frac{2}{e^2}+\el(\beta,X_0) \\
&< \frac{2}{3} + \el(\beta,Y_0) \\
&\leq \el(\alpha, Y_0)+\el(\beta,Y_0)\\
&=\el(\alpha+\beta, Y_0)
\end{align*}
and
\begin{align*}
\el(\alpha+\beta, Z_0) &=\el(\alpha, Z_0)+\el(\beta,Z_0) \\
&\geq \frac{2}{3}+\el(\beta,Z_0) \\
&> \frac{2}{e^2} + \el(\beta,W_0) \\
&\geq \el(\alpha, W_0)+\el(\beta,W_0)\\
&=\el(\alpha+\beta, W_0).
\end{align*}

The analogous inequalities must hold for small enough $a>0$ by convergence.

\end{proof}

By Lemma \ref{lem:criterion}, this implies the existence of non-convex balls in $\teich(S_{0,7})$, where $S_{\g,\p}$ is the closed surface of genus $\g$ with $\p$ points removed.

\subsection{Increasing the genus}

We modify the above construction to get a surface of genus $1$ with $4$ punctures. As before, we start with the polygon $P_a=P(1,1,1,1/a^3,1/a^2,1/a,a,1/a)$ for $0<a<1$ and take a copy $P_a^*$ of $P_a$ with reverse orientation. We think of $P_a$ as the front of the surface to be constructed and $P_a^*$ as the back. We glue each the side of $P_a$ to the corresponding side of $P_a^*$ except for the highest two horizontal sides. Call these sides $A$ and $B$ and let $A^*$ and $B^*$ be the corresponding sides of $P_a^*$. Then we glue $A$ to $B^*$ and $B$ to $A^*$ to obtain $\overline{\Phi}_a$. In other words, we glue the circle $A\cup A^*$ to $B\cup B^*$ in an orientation-reversing manner but with a half-twist. This creates a handle and a singularity of angle $4\pi$. Then we remove the $4$ singularities of angle $\pi$ from $\overline{\Phi}_a$ to obtain the half-translation surface $\Phi_a$. The curves $\alpha$ and $\beta$ are as before. We can also cut the top squares in $P_a$ and $P_a^*$ along their diagonal, rotate and glue $A\cup A^*$ to $B\cup B^*$ to obtain another useful representation of $\Phi_a$. See Figure \ref{fig:gluing}.

\begin{figure}[htbp] 
\centering
\includegraphics{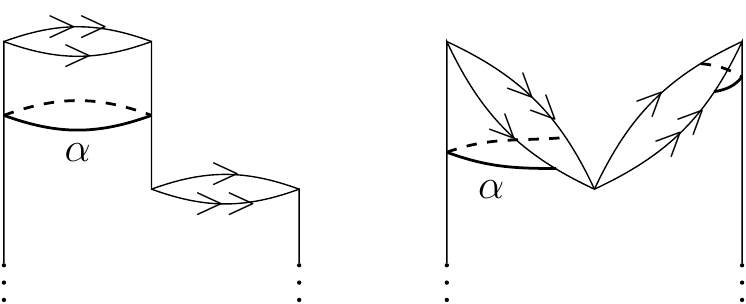}
\caption{Left: adding a handle to $\Phi_a$. Right: another representation of $\Phi_a$ obtained by cut-and-paste.}\label{fig:gluing}
\end{figure}

As in the previous subsection, we let $X_a = \geod_{-1} \Phi_a$, $Y_a = \Phi_a$, $Z_a = \geod_{\log(1/a)} \Phi_a$ and $W_a = \geod_{\log(1/a)+1} \Phi_a$. The claim is that all of these have conformal limits as $a \to 0$. 

Let $\Upsilon$ be two copies of the polygon $P(1,1,1,\infty)$ glued in the same pattern as described above, i.e.,  as in Figure \ref{fig:gluing}. Also let two copies of the polygon $P(1,0,1,\infty)$ with corresponding vertical sides glued together, the segment $[0,1]$ on the front glued to $[1,2]$ on the back, and vice versa. Denote the resulting surface $\Omega$. 

Let $Y_0 = \Upsilon \sqcup \Psi^\beta$, $X_0 = \geod_{-1} Y_0$, $Z_0= \Omega \sqcup \tau \Psi^\alpha $ and $W_0 = \geod_1 Z_0$, where $\Psi^\alpha$, $\Psi^\beta$ and $\tau$ are as in the previous subsection. 

\begin{lem}
For each $\Lambda\in \{X,Y,Z,W\}$, the surface $\Lambda_a$ converges conformally to $\Lambda_0$ as $a\to 0$.
\end{lem}
\begin{proof}
For each $\Lambda$, the convergence from the point of view of the bottom right subsurface containing $\beta$ holds for the same reasons as before.

From the point of view of $\alpha$, it is clear that $Y_a$ converges conformally to $\Upsilon$ as $a \to 0$ since any compact subset of $\Upsilon$ eventually embeds isometrically into $Y_a$. Similarly, the top left portion of $X_a$ converges conformally to $ \geod_{-1} \Upsilon$ as $a \to 0$.

The only part left to prove is that $Z_a$ and $W_a$ converge to $\Omega$ from the point of view of $\alpha$. We prove this for $Z_a$, the other case being similar. 

For $a\geq 0$ and $L>0$, let
$$
T_a^L=\{\, (x,y, \eps) \in \RR^2 \times \{+,-\}  : |x|\leq 1, -L < y\leq a^2|x| \,\} / \sim
$$

where $(x,a^2|x|,+)\sim(-x,a^2|x|,-)$ for every $x \in [-1,1]$. This is a torus with one hole obtained by gluing two $M$--shapes together. Note that $T_0^\infty = \Omega$. 

If we rescale $Z_a$ by a factor $a$ so that its left vertical chimney has circumference $4$, we see that $T_a^L$ embeds conformally into $Z_a$ provided that $1/a \geq L$. This uses the alternative gluing pattern for $\Phi_a$ with diagonal lines.

Consider the piecewise linear homeomorphism $f_a : T_0^\infty \to T_a^\infty$ defined by $$f_a(x,y,\eps) = (x,a^2|x|+y,\eps).$$ On each piece of $T_0^\infty$ where $x$ and $\eps$ have constant sign, the map $f_a$ is a vertical shear. Its dilatation tends to $1$ as $a \to 0$. 

Let $L\in(0,\infty)$. If $1/a \geq L$, then the restriction of $f_a$ to $T_0^L$ followed by the conformal embedding of $T_a^L$ into $Z_a$ provides a quasiconformal embedding with dilatation arbitrarily close to $1$. Since the subsurfaces $T_0^L$ exhaust $\Omega$, we are done.  

\end{proof}

We leave it to the reader to check that the extremal length estimates of Lemma \ref{lem:bound1}, Lemma \ref{lem:bound2} and Lemma \ref{lem:bound3} hold for this example as well. In the same way as before, we deduce that 
$$
\el(\alpha+\beta, X_a)<\el(\alpha+\beta, Y_a) \quad \mbox{and} \quad \el(\alpha+\beta, Z_a)>\el(\alpha+\beta, W_a)
$$
provided that $a$ is small enough. Hence there exist non-convex balls in $\teich(S_{1,4})$.

In the same fashion, we can further replace the $3$ punctures on the bottom right of $\Phi_a$ by a handle,  which shows that $\teich(S_{2,1})$ contains non-convex balls.

We can also produce examples in any higher topological complexity as follows. Suppose that $3\g-3+\p > 4$ and let $\mathbf{h}=\min(2,\g)$ and $\mathbf{q} = 7 - 3\mathbf{h}$. Let $\Phi_a$ be the half-translation surface constructed above of genus $\mathbf{h}$ with $\mathbf{q}$ punctures. In the bottom left corner of $\Phi_a$, we may remove $\p - \mathbf{q}$ points, cut $\g-\mathbf{h}$ horizontal slits, and glue each one back to itself in an $ABA^{-1}B^{-1}$ pattern to form a handle. The resulting half-translation surface $\wtilde\Phi_a$ has genus $\g$ and $\p$ punctures. Moreover, the conformal limits of $\geod_{-1} \wtilde\Phi_a$, $\wtilde\Phi_a$, $\geod_{\log(1/a)} \wtilde\Phi_a$ and $\geod_{\log(1/a)+1} \wtilde\Phi_a$ for the top left and bottom right subsurfaces are all unchanged. Indeed, the images of the nearly conformal embeddings used to prove conformal convergence were all disjoint from the bottom left corner. The same proof carries over and we obtain:

\begin{thm}
There exist non-convex balls in $\teich(S_{\g,\p})$ whenever $3\g - 3 + \p \geq  4$.
\end{thm}

This leaves out $5$ cases with $\dim_\CC \teich(S_{\g,\p}) = 3 \g - 3 + \p > 1$: 
$S_{0,5}$, $S_{0,6}$, $S_{1,2}$, $S_{1,3}$ and $S_{2,0}$. Note that the above strategy of proof cannot be applied to $S_{0,5}$. Indeed, we would need to split $S_{0,5}$ into two subsurfaces each containing an essential simple closed curve. But no matter how we cut $S_{0,5}$, one component is a sphere with at most $3$ holes, hence has no essential simple closed curve. Although the limiting argument does not carry over, the idea of playing a horizontal curve against a vertical curve is still fruitful.

\section{Lower complexity cases} \label{sec:lshapes}

\subsection{The Schwarz-Christoffel formula}

Consider the polygon $L_a=P(1,a,a,1)$ where $a>0$ and $P$ is the staircase-shaped polygon from section \ref{sec:staircase}. This is an $L$-shape obtained by removing the top right $a$ by $a$ square from a $(1+a)$ by $(1+a)$ square. We mark each of the $5$ internal right angles in $L_a$. Let $\alpha$ be the arc crossing the vertical leg in $L_a$, let $\beta$ be the arc crossing the horizontal leg, and let $\gamma=\alpha+\beta$. 

\begin{figure}[htbp] 
\centering
\includegraphics{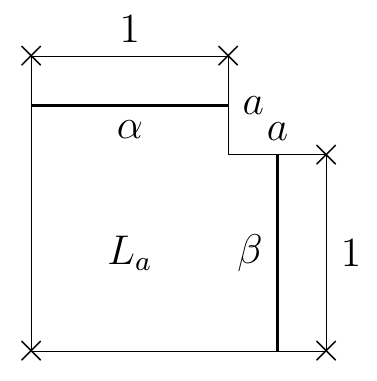}
\caption{The polygon $L_a$.}\label{fig:lshape}
\end{figure}

Since all the sides in $L_a$ are either horizontal or vertical, its double $\Phi_a$---topo\-lo\-gi\-cally a sphere with $5$ marked points---is a half-translation surface. We want to study the behavior of the extremal length of the double $\what \gamma$ of $\gamma$ in $\Phi_a$ under the Teichm\"uller flow.

\begin{lem} \label{lem:confmap}
There exists a conformal homeomorphism $h: \geod_t L_a \to M$ where $M=R_\alpha \cup R_\beta$ is a stack of two rectangles of height $1$ that line up on their right side such that 
\begin{enumerate}
\item the inverse images of corners of $M$ with interior angle $\pi/2$ are marked points in $\geod_t L_a$; 
\item $h(\alpha)$ joins the left side of $R_\alpha$ to the right side of $M$;
\item $h(\beta)$ joins the left side of $R_\beta$ to the right side of $M$.
\end{enumerate}
Moreover, we have $$\el(\what \gamma, \geod_t \Phi_a) = 2 \el(\gamma, \geod_t L_a) = 2 \area(M).$$
\end{lem}
\begin{proof}
As mentioned in Section \ref{sec:prelim}, the extremal length $\el(\what\gamma,\geod_t \Phi_a)$ is realized by a unique Jenkins-Strebel half-translation structure $\geod_t \Phi_a \to \Psi$ partitioned into two horizontal cylinders $C_\alpha$ and $C_\beta$ of height $1$ each, homotopic to the doubles $\what \alpha$ and $\what \beta$ of the arcs $\alpha$ and $\beta$. Then
$$\el(\what \gamma, \geod_t \Phi_a) = \el(C_\alpha)+\el(C_\beta) = \area(C_\alpha)+\area(C_\beta) = \area(\Psi).$$

Let $J : \geod_t \Phi_a \to \geod_t \Phi_a$ be the anti-conformal involution exchanging $\geod_t L_a$ with its mirror image. Then $J(C_\alpha)$ and $J(C_\beta)$ are disjoint cylinders homotopic to $\what \alpha$ and $\what \beta$ respectively having the same extremal length as $C_\alpha$ and $C_\beta$. By uniqueness of the extremal cylinders, the latter are invariant under $J$. It follows that $\Psi$ is also symmetric with respect to $J$. Indeed, $\overline{J^* \Psi}$ is a half-translation structure on $\geod_t \Phi_a$ partitioned into two horizontal cylinders of height $1$ homotopic to $\what \alpha$ and $\what \beta$, and is thus equal to $\Psi$ by uniqueness.

Any anti-conformal involution of a Euclidean cylinder $S^1 \times I$ which reverses the orientation of its core curve comes from a reflection of $S^1$ about a diameter. Thus $R_\alpha=C_\alpha \cap \geod_t L_a$ and $R_\beta=C_\beta \cap \geod_t L_a$ are Euclidean rectangles of height $1$ in the half-translation structure $\Psi$. Let $h: \geod_t L_a \to M$ be the restriction of the conformal isomorphism $\geod_t \Phi_a  \to \Psi$. Then $M=R_\alpha \cup R_\beta$ with $R_\alpha$ and $R_\beta$ glued isometrically along some part of their horizontal boundary.

The Gauss-Bonnet theorem tells us that $\Psi$ an angle defect of $4\pi$. Since $\Psi$ has at most $5$ cone points of angle $\pi$, it has at most one cone point of angle $3\pi$. Such a cone point has to lie on the circle of symmetry of $\Psi$, otherwise there would be two. Thus $M$ has no singularities in its interior, which means that it is really a polygon. The preimages of the right angles in $M$ by $h$ have to be marked points in $\geod_t L_a$, for after doubling $M$ the right angles give rise to $\pi$-angle singularities of $\Psi$. Since there are only $5$ marked points in $\geod_t L_a$, the rectangles have to line up on one side. If we rotate $M$ so that $R_\alpha$ is on top, then they line up on the right side, where there is no marked point separating $\alpha$ from $\beta$.

Let $\rho$ be the Euclidean metric on $M$. Then
\begin{align*}
\el(\gamma, \geod_t L_a) &= \el(h(\gamma), M) \\ & \geq \frac{\ell(h(\gamma),\rho)^2}{\area(\rho)} = \frac{(\ell(R_\alpha)+\ell(R_\beta))^2}{\area(M)} = \area(M).
\end{align*}
On the other hand, if the ratio $\ell(\gamma, \sigma)^2 / \area(\sigma)$ was strictly bigger than $\area(M)$ for some conformal metric $\sigma$ on $\geod_t L_a$, then by doubling we would get
$$
\el(\what \gamma, \geod_t \Phi_a) \geq \frac{\ell(\what \gamma, \what \sigma)^2}{\area(\what \sigma)} =  \frac{(2\ell(\gamma, \sigma))^2}{2\area(\sigma)}> 2 \area(M) = \area(\Psi),
$$
a contradiction. Alternatively, one can prove that $\rho$ is extremal using the standard length-area argument \cite{Dylan}.
\end{proof}

\begin{figure}[htbp] 
\centering
\includegraphics{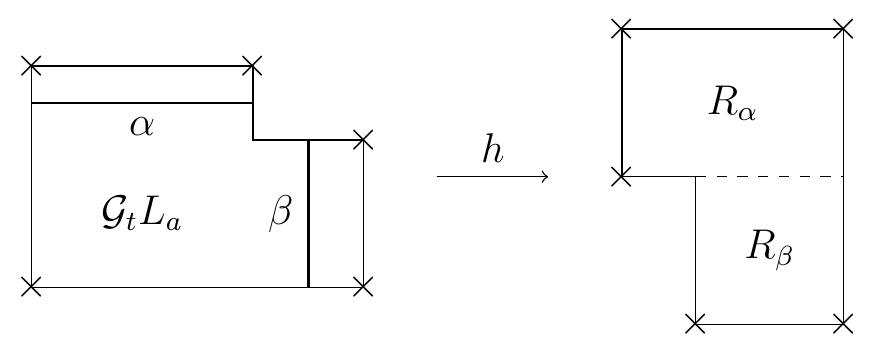}
\caption{The conformal homeomorphism in Lemma \ref{lem:confmap}}\label{fig:confmap}
\end{figure}

If we show that for some $a>0$ the function $t \mapsto \el(\gamma,\geod_t L_a)$ increases and later decreases, then the same holds for the function $t \mapsto \el(\what\gamma,\geod_t \Phi_a)$ and this implies the existence of non-convex balls in $\teich(S_{0,5})$.

Another relevant observation is that the reflection $\tau$ in the diagonal line $y=x$ maps $\geod_t L_a$ to $\geod_{-t} L_a$ anti-conformally and sends the homotopy class of $\gamma$ to itself so that the function $t \mapsto \el(\gamma,\geod_t L_a)$ is even. Therefore, all we have to show is that there exists positive $a$ and $t$ such that $\el(\gamma,\geod_0 L_a)>\el(\gamma,\geod_t L_a)$.

Let $f: \overline{\HH^2} \to \geod_t L_a$ be a conformal homeomorphism. Then $f$ extends by Schwarz reflection to a conformal homeomorphism $ f : \CHAT \to \geod_t \overline{\Phi_a}$. The pull-back $q={f}^* dz^2$ is a meromorhic quadratic differential on $\CHAT$ with a simple pole at the preimage of each marked point and a simple zero at the preimage of the inward corner in $\geod_t L_a$. Moreover, $q$ is symmetric with respect to complex conjugation. We thus have
$$
q= \frac{A(z-b)}{\Pi_{j=0}^4(z-z_j)} dz^2 = (f'(z))^2 dz^2
$$
for some $A$, $b$ and $z_j$ in $\RR$. It follows that
$$
f(z) = \sqrt{A} \int_0^z \sqrt{\frac{(\zeta-b)}{\Pi_{j=0}^4(\zeta-z_j)}} d\zeta + f(0).
$$
This is a special case of the Schwarz-Christoffel formula for conformal maps onto polygons \cite[p.10]{SCMapping}. For the formula to make sense, one has to pick a consistent choice of square root, which we can do on $\overline{\HH^2}$. 

Let $g=h \circ f: \overline{\HH^2} \to M$ where $h: \geod_t L_a \to M$ is as in Lemma \ref{lem:confmap}. By the same reasoning as above, $g$ has the form
$$
g(z) = C \int_0^z \sqrt{\frac{(\zeta-p)}{\Pi_{j=0}^4(\zeta-z_j)}} d\zeta + D
$$
for some constants $p$, $C$ and $D$. The area of $M$ can then be recovered from its side lengths, which are integrals of the above form.

\begin{remark}
One can use the Schwarz-Christoffel formula to prove the first part of Lemma \ref{lem:confmap}. Indeed, for any choice of $p \in \RR $, the map
$$
G_p(z) = \int_0^z \sqrt{\frac{(\zeta-p)}{\Pi_{j=0}^4(\zeta-z_j)}} d\zeta
$$
is a conformal homeomorphism from $\overline{\HH^2}$ to a polygon with angle $\pi/2$ at each vertex $G_p(z_j)$ and angle $3\pi/2$ at $G_p(p)$. Suppose that 
\begin{equation} \label{eq:ineq}
|G_{z_1}(z_0)-G_{z_1}(z_1)|\leq |G_{z_1}(z_1)-G_{z_1}(z_2)|.
\end{equation}
Then by the intermediate value theorem, there exists a point $p$ between $z_1$ and $z_2$ such that
$$
|G_{p}(z_0)-G_{p}(z_1)|=|G_{p}(p)-G_{p}(z_2)|.
$$ 
Indeed, $|G_{p}(z_0)-G_{p}(z_1)|$ is bounded away from zero for $p \in [z_1,z_2]$ whereas $|G_{p}(p)-G_{p}(z_2)|$ tends to zero as $p \to z_2$. If the reverse of inequality (\ref{eq:ineq}) holds, then there is a $p$ between $z_0$ and $z_1$ such that
$$
|G_{p}(z_0)-G_{p}(p)|=|G_{p}(z_1)-G_{p}(z_2)|.
$$
In either case, after rescaling we get that $G_p(\overline{\HH^2})$ is a stack of two rectangles of height $1$.
\end{remark} 

The problem of calculating $\el(\gamma, \geod_t L_a)$ has now been reduced to finding the correct parameters $z_0,...,z_4$, $b$ and $p$ (all of which depend on $a$ and $t$). The Schwarz-Christoffel Toolbox \cite{Driscoll} for MATLAB is designed to solve this parameter problem numerically. We used this to compute $\el(\gamma, \geod_t L_a)$ for $a=1/4$ at $10^3+1$ equally spaced values of $t$ in the interval $[-0.275,0.275]$ and obtained Figure \ref{fig:plot}.

\begin{figure}[htp]
\includegraphics[scale=.7]{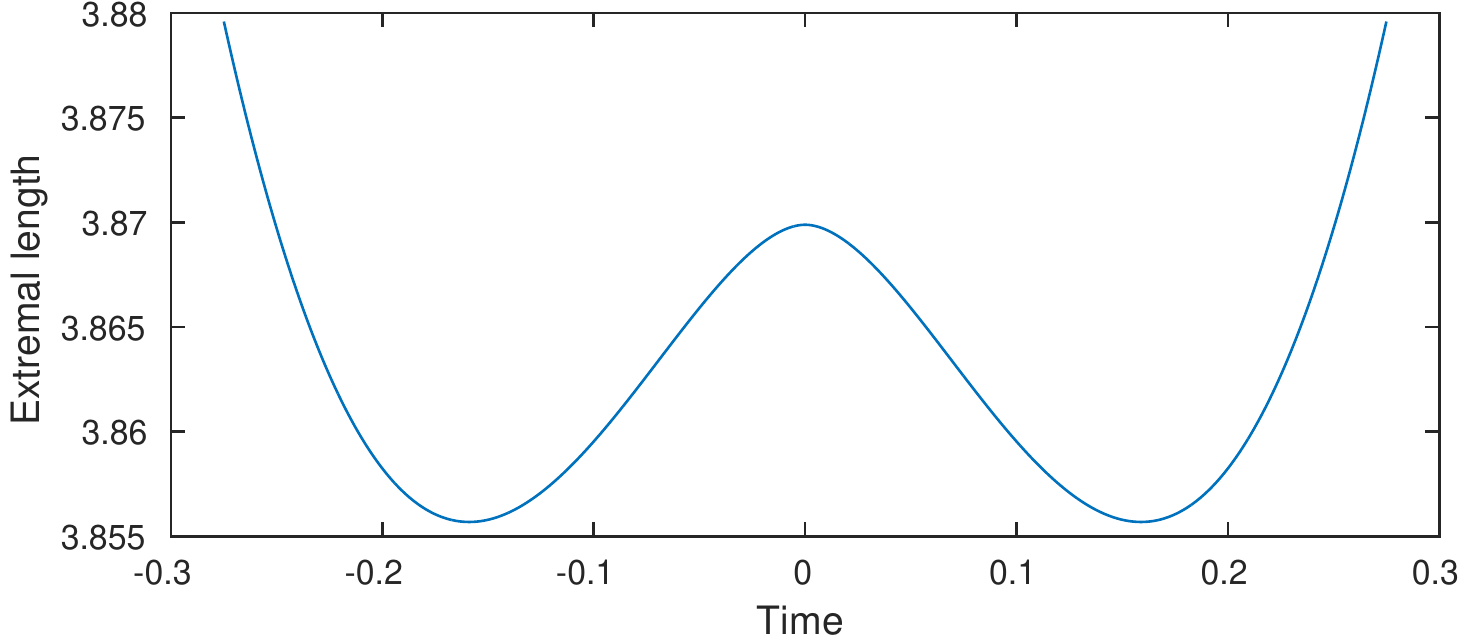}
\caption{Graph of $t \mapsto \el(\gamma, \geod_t L_a)$ for $a=1/4$.}
\label{fig:plot}
\end{figure}

The figure clearly shows a decrease from time $t=0$ to $t\approx 0.159$. However, the Schwarz-Christoffel Toolbox does not come with any certified error estimates. Moreover, the apparent decrease of extremal length is rather small: it drops from about $3.87$ to $3.856$, which represents less than $1\%$ decrease. 

In order to turn this into a rigorous proof, we do the following. We take the approximate parameters provided by the SC Toolbox, then compute the corresponding Schwarz-Christoffel integrals numerically but with certified precision. Since the initial parameters are inexact, the images of the Schwarz-Christoffel maps are not the polygons we expect, but we can estimate how far away they are from the correct polygons and deduce bounds for extremal length.

One way to get rigorous bounds on a numerical result is to use interval arithmetic. Roughly speaking, interval arithmetic means that instead of rounding to the nearest representable number, the computer keeps track of correct lower and upper bounds for every operation, yielding a true interval in which the result of a calculation lies.

There exist packages that do numerical integration using interval arithmetic. However, we did not find any that can handle improper integrals. We thus wrote a program in Sage \cite{sage} to compute lower and upper bounds on the integrals needed using interval arithmetic. The Sage worksheet and its output are available at \url{http://github.com/maxforbou/non-convex-balls}.

\subsection{Rigorous bounds}
  
Let $k=5.27110734472$, let 
$$
f(z)=\int_0^z \frac{d\zeta}{\sqrt{\zeta(\zeta^2-1)(\zeta^2-k^2)}}
$$
and let $X=f(\overline{\HH^2})$ with marked points at $0=f(0)$, $f(\pm 1)$ and $f(\pm k)$. The polygon $X$ is an $L$-shape with angle $\pi/2$ at the marked points and angle $3\pi/2$ at $f(\infty)$. Furthermore, $X$ is symmetric about the diagonal line $y=x$ since the function under the square root is odd. Thus, $X$ is a rescaled copy of $L_a$, where
$$
a=\frac{|f(0)-f(1)|}{|f(1)-f(k)|}-1.
$$ 

We want to get rigorous bounds on both the shape of $X$ and the extremal length of $\gamma$ in $X$. The first thing we need to compute is the integral
$$
|f(0)-f(1)|=\left|\int_0^1 \frac{dx}{\sqrt{x(x^2-1)(x^2-k^2)}}\right|=\int_0^1 \frac{dx}{\sqrt{|x(x^2-1)(x^2-k^2)|}}.
$$
The main observation is that the integrand $$F(x)=\frac{1}{\sqrt{|x(x^2-1)(x^2-k^2)|}}$$ is logarithmically convex (hence convex) on $(0,1)$.

\begin{lem}
Suppose that $z_0 < z_1 < z_2 < z_3 < z_4$. Then the function
$$
F(x)= \prod_{j=0}^4|x-z_j|^{-1/2}
$$
is log-convex between any two consecutive $z_j$'s.
\end{lem}
\begin{proof}
We compute
$$
(\log F)'(x) = -\frac12 \left(\sum_{j=0}^4 \frac{1}{x-z_j} \right)
$$
and 
$$
(\log F)''(x) = \frac12 \left(\sum_{j=0}^4 \frac{1}{(x-z_j)^2} \right)>0.
$$
\end{proof}

Therefore, for any compact subtinterval $I \subset (0,1)$ and any partition $\{x_{-n},\ldots,x_n\}$ of $I$ we have
\begin{align*}
\sum_{j=-n}^{n-1} (x_{j+1} - x_j)F\left(\frac{x_j+x_{j+1}}{2}\right) &\leq \int_I F(x)\,dx \\
 \leq &\sum_{j=-n}^{n-1} (x_{j+1} - x_j)\left(\frac{F(x_j)+F(x_{j+1})}{2}\right)
\end{align*}
by the trapezoid rule. We choose the partition $\{x_{-n},\ldots,x_n\}$ using the tanh-sinh quadrature \cite{tanhsinh} which is well-adapted for this type of singular integral. On a bounded interval $(a,b)$ the quadrature points are defined as
$$
x_j = \frac{(a+b)}{2} + \frac{(b-a)}{2}\tanh\left(\frac{\pi}{2}\sinh(j\Delta)\right)
$$
where $\Delta>0$ is a step size to be determined together with $n$. In this case we took $\Delta=2^{-13}$ and $n=2^{15}$. 

Let $\delta=x_{-n}=1-x_n$ where the $x_j$'s are sample points for the interval $(0,1)$. An elementary calculation shows that
$$
0 \leq \int_0^\delta \frac{dx}{\sqrt{x(1-x^2)(k^2-x^2)}} \leq \frac{2 \sqrt{\delta}}{\sqrt{(1-\delta^2)(k^2-\delta^2)}}.
$$
and 
$$
0 \leq \int_{1-\delta}^1 \frac{dx}{\sqrt{x(1-x^2)(k^2-x^2)}} \leq \frac{2\sqrt{\delta}}{\sqrt{(1-\delta)(2-\delta)(k^2-1)}}.
$$ 

Adding the lower bounds for each of the three subintervals $[0,\delta]$, $[\delta, 1-\delta]$ and $[1-\delta , 1]$ yields a certified lower bound on $|f(0)-f(1)|$, and similarly for upper bounds. We use the same method to estimate $|f(1)-f(k)|$.

In order to compute the extremal length $\el(\gamma,X)$, we consider the conformal homeomorphism
$$
g(z)=-i \int_0^z \frac{d\zeta}{\sqrt{(\zeta^2-1)(\zeta^2-k^2)}}
$$
between $\overline{\HH^2}$ and a rectangle $X'$ with marked points at $0=g(0)$, $g(\pm 1)$ and $g(\pm k)$. Then $g\circ f^{-1} : X \to X'$ is a conformal homeomorphism preserving the marked points so that $\el(\gamma,X)=\el(\gamma,X')$. Since the above integrand is even, $g(0)$ subdivides $X'$ into two congruent rectangles. After rescaling $X'$ to have height $2$, the extremal length is given by area. This means that 
$$
\el(\gamma,X)=\el(\gamma,X')=2\frac{|g(1)-g(k)|}{|g(0)-g(1)|}.
$$ We get rigorous bounds on $|g(0)-g(1)|$ and $|g(1)-g(k)|$ with the same method as for $f$. The results are compiled in Table \ref{table1}.

\begin{table}[htp] 
\centering
\begin{tabular}{|c|c|c|c|}
\hline
 & lower bound & upper bound & approximation \\ 
\hline 
$|f(0)-f(1)|$ &  0.500482492919 & 0.500482504323 &  0.500482496721 \\
\hline
$|f(1)-f(k)|$ &  0.400385993317 & 0.400386005494 &  0.400385997377   \\
\hline
$|g(0)-g(1)|$ & 0.300738235179 & 0.300738239980 &  0.30073823678   \\
\hline
$|g(1)-g(k)|$ &  0.581911579444 & 0.581911593793 & 0.581911584228  \\
\hline
$\el(\gamma,X)$ & 3.86988751070 & 3.86988766789 & 3.86988758368  \\
\hline
\end{tabular}
\caption{Certified bounds on the side lengths of $X$ and $X'$. The last column shows the corresponding value calculated with Sage's \texttt{nintegral} routine.} \label{table1}
\end{table}

We now estimate extremal length after stretching $X$. Let the parameters $z_0$, $z_1$, $z_2$, $p$, $z_3$ and $z_4$ be equal to $-3.33297982345$, $-0.26873921366$, $0$, $0.17317940636$,
$1$ and $2.94288195633$ respectively. Then let
$$
\phi(z)=\int_{0}^z \prod_{j=0}^4 (\zeta-z_j)^{-1/2}\, d\zeta,
$$
$$
\psi(z)=-i \int_{0}^z (\zeta-p)^{1/2}\prod_{j=0}^4 (\zeta-z_j)^{-1/2}\, d\zeta,
$$
$Y=\phi(\overline{\HH^2})$ and $Y'=\psi(\overline{\HH^2})$. The polygon $Y$ is meant to be close to a rescaled version of $\geod_t X$ for $t\approx 0.159$ whereas $Y'$ is a stack of two rectangles of nearly the same height, which we use to estimate $\el(\gamma,Y)$.

Since the integrand $\prod_{j=0}^4 |x-z_j|^{-1/2}$ is convex, we may use the trapezoid rule to compute the side lengths of $Y$. There are also elementary estimates near the poles like before. 

For $Y'$ the integrand is convex on each interval of continuity not adjacent to $p$. Indeed,
if $$G(x)= |x-p|^{1/2}\prod_{j=0}^4 |x-z_j|^{-1/2}$$ then 
$$
2(\log G)''(x)=  \sum_{j=0}^4 \frac{1}{(x-z_j)^2} - \frac{1}{(x-p)^2}
$$ 
which is positive when $x<z_2$ and when $x>z_3$. We can thus apply the trapezoid rule with tanh-sinh quadrature to bound the side lengths of $Y'$ not adjacent to $\psi(p)$. The length 
$$
|\psi(z_4)-\psi(z_0)| = |\psi(z_4)-\psi(\infty)| + |\psi(\infty)-\psi(z_0)|
$$
is a little bit different since we need to compute integrals over two half-infinite intervals. We use another doubly exponential quadrature on these intervals given by
$$
x_j = \exp\left(\frac{\pi}{2}\sinh(j \Delta)\right)
$$
for the interval $(0,\infty)$. To estimate the area lost by truncating away from infinity, note that for $x> 2z_4-p$ we have $|x-p|< 2|x-z_4|$ as well as $|x-z_j|\geq |x-z_4|$ for each $j$. It follows that  $G(x) \leq \sqrt{2} |x-z_4|^{-2}$ and hence
$$
\int_a^\infty G(x)\,dx \leq \sqrt{2} |a-z_4|^{-1} 
$$
provided that $a \geq 6$. Similarly, we have
$$
\int_{-\infty}^b G(x)\,dx \leq \sqrt{2} |b-z_0|^{-1} 
$$
provided that $b \leq -7$. 

The polygon $Y'$ is not exactly a stack of two rectangles of the same height, but we can still use it to estimate  $\el(\gamma,Y)= \el(\gamma,Y')$. Using the Euclidean metric on $Y'$ yields the lower bound
$$
\el(\gamma,Y') \geq \frac{\ell(\gamma)^2}{\area(Y')} = \frac{(|\psi(z_0)-\psi(z_1)|+|\psi(z_3)-\psi(z_4)|)^2}{\area(Y')}.
$$ 
Moreover, the sum of the extremal lengths of the horizontal rectangles $R_\alpha$ and $R_\beta$ in $Y'$ is an upper bound for the extremal length:
$$
\el(\gamma,Y') \leq \frac{|\psi(z_0)-\psi(z_1)|}{|\psi(z_1)-\psi(z_2)|}+\frac{|\psi(z_3)-\psi(z_4)|}{|\psi(p)-\psi(z_3)|}.
$$

The last caveat is that $Y$ does not lie exactly along the Teichm\"uller geodesic through $X$. Let $$a=\frac{|f(0)-f(1)|}{|f(1)-f(k)|}-1$$ and $$K=\frac{|\phi(z_0)-\phi(z_1)|}{|\phi(z_3)-\phi(z_4)|}$$ and consider the polygon $Z=P(K,a,Ka,1)$. Then up to rescaling $Z = \geod_t X$ for $t=\frac{1}{2}\log K$. Divide each of $Y$ and $Z$ into three rectangles with sides parallel to the coordinate axes and let $h:Y \to Z$ be the homeomorphism which is affine on each subrectangle. Then $h$ preserves the marked points and
$$
\frac{1}{C}\el(\gamma,Y) \leq \el(\gamma,Z) \leq C \el(\gamma,Y)
$$
where $C\geq \exp(2 d(Y,Z))$ is the dilatation of $h$. Note that $C$ can be expressed in terms of the aspect ratios of the three subrectangles in $Y$ and $Z$. The resulting bounds are shown in Table \ref{table2}.

\begin{table}[htp] 
\centering
\begin{tabular}{|c|c|c|c|}
\hline
 & lower bound & upper bound & approximation \\ 
\hline 
$|\phi(z_0)-\phi(z_1)|$ & 1.036823405576 & 1.036823443983 &  1.03682341838     \\
\hline
$|\phi(z_1)-\phi(z_2)|$ & 0.943128409696 &  0.943128430640 &  0.943128416679      \\
\hline
$|\phi(z_2)-\phi(z_3)|$ & 1.296029251902 & 1.296029284584 &  1.2960292628      \\
\hline
$|\phi(z_3)-\phi(z_4)|$ & 0.754502722746 & 0.754502742641 &  0.754502729379    \\
\hline 
$|\psi(z_0)-\psi(z_1)|$ & 1.068955145751 & 1.068955175385 &  1.06895515563    \\ 
\hline
$|\psi(z_1)-\psi(z_2)|$ & 0.512964353079 & 0.512964364188 &    0.512964356783    \\
\hline
$|\psi(z_3)-\psi(z_4)|$ & 0.908877581965 & 0.908877603159  &  0.908877589032      \\
\hline
$|\psi(z_4)-\psi(z_0)|$ &  1.025928700631 & 1.025928738891 &   1.02592871356 \\
\hline
$\el(\gamma,Y)$ &  3.855692084405 & 3.855692498209 & 3.85569234685  \\
\hline
$\exp(2 d(Y,Z))$ & -- & 1.000000357759 & -- \\ 
\hline
$\el(\gamma,Z)$ &  3.855690704998 & 3.855693877617 &  3.85569234685  \\
\hline
\end{tabular}
\caption{Certified integrals after stretching} \label{table2}
\end{table}

We thus have
$$
\el(\gamma, \geod_t X) = \el(\gamma, Z) < 3.8557 < 3.8698 < \el(\gamma, X),
$$
from which we conclude that $\teich(S_{0,5})$ contains non-convex balls.

\subsection{Remaining cases}

Adding an artificial marked point on the boun\-dary of $X$ between $f(1)$ and $f(k)$ (the right-most side of $X$) does not change the extremal length of $\gamma$ at any time. After doubling, this shows the existence of a non-convex ball in $\teich(S_{0,6})$. 

Recall that there are isometries $\teich(S_{0,5}) \cong \teich(S_{1,2})$ and $\teich(S_{0,6}) \cong \teich(S_{2,0})$ arising from the hyperelliptic involutions on $S_{1,2}$ and $S_{2,0}$. This shows that there exist non-convex balls in those two cases as well.

To treat the torus with $3$ punctures, we can cut horizontal slits of length $s>0$ at two punctures in the double of $X$ then glue the two slits together to form a handle. As $s \to 0$, the extremal length of the double $\what \gamma$ of $\gamma$ on the 3 times punctured torus converges to its value on the double of $X$. The same is true after applying the Teichm\"uller flow $\geod_t$ for any $t$. It follows that if $s$ is small enough, then the resulting geodesic in $\teich(S_{1,3})$ exhibits an increase of extremal length followed by a decrease. This completes the proof of Theorem \ref{thm:mainthm}.

\bibliographystyle{amsalpha}
\providecommand{\bysame}{\leavevmode\hbox to3em{\hrulefill}\thinspace}
%\providecommand{\MR}{\relax\ifhmode\unskip\space\fi MR }
%\providecommand{\MRhref}[2]{%
%  \href{http://www.ams.org/mathscinet-getitem?mr=#1}{#2}
%}
\providecommand{\href}[2]{#2}

\end{document}